\documentclass{aims}
\usepackage{amsmath}
  \usepackage{paralist}
  \usepackage{graphics} 
  \usepackage{epsfig} 
\usepackage{graphicx}  \usepackage{epstopdf}
 \usepackage[colorlinks=true]{hyperref}

\usepackage{multicol}

\usepackage{geometry}
\usepackage{url}
\usepackage{array}
\usepackage{bm}

\usepackage{layout}

\usepackage{color}
\usepackage{gauss} 
\usepackage{enumerate}   

\usepackage{graphicx}
\usepackage{epsfig}
\usepackage{subfig}
\usepackage{grffile}
\usepackage{tikz}
\usepackage{pgfplots}
\usepackage{pgfplotstable}
\usepackage{cleveref}

\newcommand{\Id}{{\mathsf{Id}}}

\newcommand{\R}{{\mathbb R}}

\newcommand{\trace}{\mathrm{trace}}
\newcommand{\Was}{\mathrm{W}}

\newcommand{\dWas}{\Was_{2}}
\newcommand{\Cb}{\mathbf{C}}
\newcommand{\mb}{\mathbf{m}}
\newcommand{\Gb}{\mathbf{G}}
\newcommand{\bSigma}{\boldsymbol{\Sigma}}

\newcommand{\pset}{\Theta} 
\newcommand{\param}{\theta} 
\newcommand{\norm}[1]{\lVert#1\rVert}

\usepackage{algorithm}
\usepackage{algpseudocode}
\hypersetup{urlcolor=blue, citecolor=red}

\def\currentvolume{X}
 \def\currentissue{X}
  \def\currentyear{200X}
   \def\currentmonth{XX}
    \def\ppages{X--XX}

\newtheorem{theorem}{Theorem}[section]
\newtheorem{corollary}[theorem]{Corollary}

\newtheorem{lemma}[theorem]{Lemma}

\newtheorem{example}[theorem]{Example}

\theoremstyle{definition}
\newtheorem{definition}[theorem]{Definition}
\newtheorem{remark}[theorem]{Remark}

\title[Certified and fast computations with shallow covariance kernels] 
      {Certified and fast computations with shallow covariance kernels}

\author[D. Kressner, J. Latz, S. Massei, E. Ullmann]{}


 \email{daniel.kressner@epfl.ch}
 \email{jl2160@cam.ac.uk}
 \email{s.massei@tue.nl}
 \email{elisabeth.ullmann@ma.tum.de}

\thanks{Jonas Latz and Elisabeth Ullmann acknowledge the support by Deutsche Forschungsgemeinschaft (DFG) and Technische Universit\"at M\"unchen (TUM) through the TUM International Graduate School of Science and Engineering (IGSSE) within the project 10.02 BAYES.  The work of Stefano Massei has been supported by the SNSF research project \emph{Fast algorithms from low-rank updates}; grant number: 200020\_178806.}

\thanks{$^*$ Corresponding author: Jonas Latz}

\begin{document}
\maketitle

\centerline{\scshape Daniel Kressner}
\medskip
{\footnotesize
 \centerline{MATH-ANCHP, \'Ecole Polytechnique f\'ed\'erale de Lausanne}
   \centerline{1015 Lausanne, Switzerland}
} 

\medskip
\centerline{\scshape Jonas Latz$^*$}
\medskip
{\footnotesize
 \centerline{Department of Applied Mathematics and Theoretical Physics, University of Cambridge}
   \centerline{Cambridge CB3 0WA, United Kingdom}
} 

\medskip
\centerline{\scshape Stefano Massei}
\medskip
{\footnotesize
 \centerline{Department of Mathematics and Computer Science, Eindhoven University of Technology}
   \centerline{5612 AZ Eindhoven, Netherlands}
} 

\medskip

\centerline{\scshape Elisabeth Ullmann}
\medskip
{\footnotesize
 \centerline{ Department of Mathematics, Technical University of Munich}
   \centerline{85748 Garching, Germany}
}

\bigskip

 \centerline{(Communicated by the associate editor name)}

\begin{abstract}
Many techniques for data science and uncertainty quantification demand efficient tools to handle Gaussian random fields, which are defined in terms of their mean functions and covariance operators. 
Recently, parameterized Gaussian random fields have gained increased attention, due to their higher degree of flexibility.
However, especially if the random field is parameterized through its covariance operator, classical random field discretization techniques fail or become inefficient.
In this work we introduce and analyze a new and certified algorithm for the low-rank approximation of a parameterized family of covariance operators which represents an extension of the \emph{adaptive cross approximation}  method for symmetric positive definite matrices. The algorithm relies on an affine linear expansion  of the covariance operator with respect to the parameters, which needs to be computed in a preprocessing step using, e.g., the empirical interpolation method. We discuss and test our new approach for isotropic covariance kernels, such as Mat\'ern kernels. The numerical results demonstrate the advantages of our approach in terms of computational time and confirm that the proposed algorithm provides the basis of a fast sampling procedure for parameter dependent Gaussian random fields.
\end{abstract}

\section{Introduction} \label{Subs_Problem}
Deep neural networks (DNNs) play a fundamental role in modern machine learning and artificial intelligence, see \cite{Higham2019} for an overview. 
They are an example for a \emph{deep model}, which is constructed by composing a number of mathematical models (`layers'). By this construction, deep models allow for much more flexibility in data-driven applications than offered by a standard model.
The \emph{layers} in deep models are often rather simple, like artificial neurons in DNNs or linear functions. 
Hence, the computational cost when \emph{training} a DNN is caused by a large number of layers, not by the complexity of the single models.
In other situations, however, the layers themselves are highly complex. 
Here, already the training of a \emph{shallow model} -- consisting of few layers -- may require an immense computational cost. 

In this article, we consider a specific shallow model with two layers. The outer layer is a  Gaussian process on a compact domain with a covariance kernel containing unknown parameters. 
The inner layer is a probability measure that describes the distribution of the unknown parameters in the covariance kernel. 
We (and, e.g., \cite{garriga-alonso2018deep}) refer to such a shallow model as a \emph{shallow Gaussian process} in contrast to \emph{deep Gaussian processes} that recently gained popularity in the literature; see \cite{Damianou13,Dunlop2018,Emzir2019}. 
Shallow Gaussian processes appear especially in the uncertainty quantification literature, see \cite{Dunlop2017, LATZ2019FastSa, Minden2017,Sraj2016,Wikle2017}, where they are used to model random, function-valued inputs to partial differential equations. Additionally, they are applicable in a variety of statistical or machine learning tasks; see, e.g., \cite{Haug2019}  and Chapter 5 of \cite{Rasmussen2006}.

\begin{figure}[thb]
\centering
\begin{tikzpicture}[scale=0.67]

\draw (-7,0.5) node[anchor = south,solid] {$\bm \param \sim \mu'$};
\draw (-7,-0.5) node[anchor = north,solid] {inner layer};
\draw (1,-0.5) node[anchor = north,solid] {outer layer};

\fill[black] (-7,0) circle (8pt);
\draw[line width=1.5,->] (-7,0) -- (0,0);
\fill[black] (1,0) circle (8pt);
\draw[line width=1.5,->] (1,0) -- (8,0);
\fill[black] (9,0) circle (8pt);

\draw (1,0.5) node[anchor = south,solid] {$X \sim M(\cdot|\bm \param)$};
\draw (9,0.5) node[anchor = south,solid] {$X \sim \mu$};
\draw (9,-0.5) node[anchor = north,solid] {output};

\end{tikzpicture}
\caption{Illustration of a shallow Gaussian process. On the inner layer, a random variable $\bm \param$ with law $\mu'$ is sampled. On the outer layer, the process $X \sim M(\cdot|\bm \param)$ is sampled. Then, $X$ follows the distribution of the shallow Gaussian process  $\mu$.} 
\label{Figure_visual}
\end{figure}
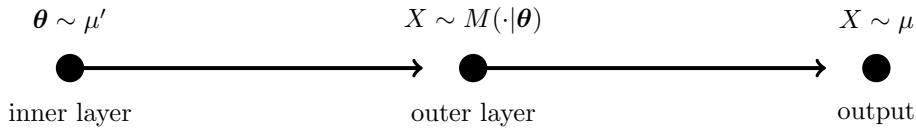

In practice, the Gaussian process on the outer layer is discretized with a very large number of degrees of freedom  $n \in \mathbb{N}$. 
The discretization induces a parameterized Gaussian measure on $(\mathbb{R}^n, \mathcal{B}\mathbb{R}^n)$ of the form
$$
M(\cdot | \bm\param) := \mathrm{N}(\mb(\bm\param), \Cb(\bm\param)) \qquad (\bm\param \in \pset).
$$
Here $\bm\param$ refers to the unknown  parameters that are contained in the parameter space $\pset{}$, which is associated with some $\sigma$-algebra $\mathcal{A}$.
Moreover, $\mb: \pset \rightarrow \mathbb{R}^n$ is a measurable function and
$\Cb: \pset \rightarrow \mathrm{SPSD}(n)$ is a measurable function mapping parameters $\bm\param \in \pset$ to the set of \emph{symmetric positive semidefinite} (SPSD) matrices  in $\mathbb{R}^{n \times n}$.
By this assumption, $M$ is a Markov kernel mapping from $(\pset, \mathcal{A})$ to $(\mathbb{R}^n, \mathcal{B}\mathbb{R}^n)$.

The inner layer is represented by a probability measure $\mu'$ on $(\pset, \mathcal{A})$.
The shallow Gaussian process is then defined as the composition of $\mu'$ and $M$: 
$$\mu := \mu'M := \int_{\pset} M(\cdot| \param) \mu'(\mathrm{d}\bm\param) = \int_{\pset} \mathrm{N}(\mb(\bm\param), \Cb(\bm\param)) \mu'(\mathrm{d}\bm\param) =:\mu'\mathrm{N}(\mb, \Cb);$$
which is a well-defined  probability measure on $(\mathbb{R}^n, \mathcal{B}\mathbb{R}^n)$.
In \Cref{Figure_visual}, we visualize the simple graph structure behind the shallow Gaussian process $\mu$.
We refer to a measure $\mu$ of the described form as a \emph{hierarchical measure}. 

To employ $\mu$ in a statistical or learning task, we have to sample from $M(\cdot | \bm\param)$ for a large number of parameter samples $\bm\param \in \pset$. 
In general, sampling requires a Cholesky factorization or spectral decomposition of $\Cb(\bm\param)$; see, e.g., \cite[p. 1873]{Graham2018}.
We assume here that $\Cb(\bm\param)$ is a dense matrix.
Then, sampling from $M(\cdot|\bm\param)$ in the standard way comes at a cost of $\mathcal{O}(n^3)$ operations for each $\bm\param \in \pset$.
Since $n$ is assumed to be very large, performing these tasks  for many $\bm\param \in \pset$ becomes computationally infeasible.

For fixed $\bm\param$, applying \emph{adaptive cross approximation (ACA)} to $\Cb(\bm\param)$ has proven effective in accelerating the computation of a (partial) Cholesky factorization \cite{harbrecht,Harbrechtetal:2015,Schaefer2017}. This is closely related to the Nystr\"om method \cite{williams2001using} which also builds approximations of a covariance matrix using a subset of its columns. 
Other than ACA, techniques employed for the fast generation of random fields are based on, e.g., circulant embedding  \cite{Bachmayr2019,Dietrich1997,Graham2018}, H-matrices \cite{Feischl2017,Khoromskij2009,Minden2017},   fractional stochastic partial differential equations \cite{Khrist2019,Lindgren2011}, or  fast computations of truncated Karhunen-Lo\`eve expansions  \cite{Rizzi2018,Saibabaetal:2016,SchwabTodor:2006}.
 
The main contribution of this article is a new variant, \emph{parameter-dependent ACA}, that aims at approximating $\Cb( \cdot )$ simultaneously  across the whole parameter range. 
For this purpose, we replace $\Cb(\cdot)$ by some low-rank approximation $\widehat{\Cb}(\cdot)$ for which sampling can be done computationally fast, even if $n$ is large. We define the corresponding \emph{approximate outer layer Markov kernel} $$\widehat{M}(\cdot|\bm\param) := \mathrm{N}(\mb(\bm\param), \widehat{\Cb}(\bm\param)) \qquad (\bm\param \in \pset)$$ and shallow Gaussian process $\widehat{\mu} := \mu'\widehat{M}$.
In our method, $\widehat{\Cb}(\cdot)$ is constructed by a greedy procedure that controls the errors in the \emph{Wasserstein} distance $W_2$. 
To be more precise, we aim at an approximation that is accurate in terms of the errors
\begin{equation} \label{Eq_Was_intro}
 \dWas\left(\mu, \widehat{\mu}\right), \qquad \qquad \dWas\left(M(\cdot|\bm\param), \widehat{M}(\cdot|\bm\param)\right) \quad (\bm\param \in \pset).
\end{equation}
This is a significant improvement compared to the heuristic algorithm proposed by Latz et al. \cite{LATZ2019FastSa}.
Approximation methods that aim at controlling the approximation error are sometimes called \emph{certified}. Further contributions of this work are the following:
\begin{itemize}
\item We propose and discuss techniques to approximate parameterized covariance matrices by linearly separable expansions, as required by ACA and other low-rank techniques,
\item We analyze the robustness with respect to loss of positive semi-definiteness of the new parameter-dependent ACA and we demonstrate that it has linear cost $\mathcal{O}(n)$,
\item We show computational experiments in which the techniques are used to construct low-rank approximations of matrices generated from Mat\'ern covariance kernels; see, e.g., \cite[Subsection 2.10]{stein}.
\end{itemize}

This work is organized as follows. In \Cref{Sec_erroranaly}, we compute upper bounds for the Wasserstein distances in  \cref{Eq_Was_intro}. 
These will be used for error control in the parameter-dependent ACA that we propose in \Cref{Sec_aca}. 
In \Cref{Sec_IsotrCKer}, we give examples for isotropic covariance kernels and discuss approximation strategies that lead to linearly separable covariance kernels.
Finally, we verify our theoretical findings in numerical experiments in \Cref{Sec_NumExp} and conclude the work in \Cref{Sec_Concl}.

\section{Error analysis} \label{Sec_erroranaly}

As discussed in \Cref{Subs_Problem}, the approximations $\widehat{\mu}, \widehat{M}$ of $\mu, M$ are obtained by replacing the parameterized covariance matrices by low-rank approximations.  In this section, we quantify the error that is introduced by such an approximation. 

Various distances for measures have been proposed, such as the total variation distance, the Hellinger distance, the Prokhorov metric, the Kullback--Leibler divergence, and the Wasserstein distance; we refer to 
Gibbs and Su \cite{Gibbs2002} for an overview. 
Let $\bm\param \in \pset$.
If $\widehat{\Cb}(\bm\param)$ is a low-rank approximation of ${\Cb}(\bm\param)$, typically $\mathrm{img}({{\Cb}(\bm\param)^{1/2}}) \neq \mathrm{img}({\widehat{\Cb}(\bm\param)^{1/2}})$.
In this case, $M(\cdot|\bm\param)$ and $\widehat{M}(\cdot|\bm\param)$ are singular measures, which makes the total variation distance and equivalent metrics unsuitable: for singular measures the total variation distance is always equal to 2.
The Wasserstein distance $W_p$, however, is suitable, as it is closely related to weak convergence; see \cite[Theorem 6.9]{Villani2009}.
\begin{definition}[Wasserstein]
Let $\nu, \widehat{\nu}$ be two probability measures on $(\mathbb{R}^n, \mathcal{B}\mathbb{R}^n)$. We define $\mathrm{Coup}(\nu,\widehat{\nu})$ to be the \emph{set of couplings} of $\nu, \widehat{\nu}$, i.e. 
 the set of probability measures $H$ on $(\mathbb{R}^{2n}, \mathcal{B}\mathbb{R}^{2n})$, such that 
$
 H(B \times \R^n ) = {\nu}(B)$, $H(\R^n \times B) = \widehat{\nu}(B)$ $(B \in \mathcal{B}\R^n).
$
Given $p \in [1, \infty)$, the \emph{Wasserstein(--$p$) distance} between $\nu$ and $\widehat{\nu}$ is defined by
$$
\Was_{p}(\nu, \widehat{\nu}) := \left(\inf_{H' \in \mathrm{Coup}(\nu,\widehat{\nu})} \int_{\mathbb{R}^{2n}} \|X_1 - X_2 \|_p^p H'(\mathrm{d}X_1, \mathrm{d}X_2)\right)^{1/p},
$$
provided that this integral is well-defined.
\end{definition}
The Wasserstein distance $W_p$ is motivated by the idea of optimal transport. It is defined as the cost of an optimal transport between the two probability measures, measured in the $p$-norm.
The distance is well-defined if the $p$th absolute moment of both  measures is finite.
This is for instance the case for measures with light tails, such as the Gaussian measures that we encounter throughout \Cref{Subsec_fixed_param}.
For more details on the Wasserstein distance, we refer to the book by Villani \cite{Villani2009}.

Throughout the rest of this paper, we assume, without loss of generality that the parameterized Gaussian measures have mean zero, i.e., $\mb \equiv 0$.

\subsection{The fixed-parameter case} \label{Subsec_fixed_param} We commence with the analysis of the Wasserstein distance $W_2$ for two Gaussian measures
\begin{align*}
\kappa := \mathrm{N}(0,C) := \mathrm{N}(0,\Cb(\bm\param)) := M(\cdot|\bm\param), \qquad 
\widehat{\kappa} := \mathrm{N}(0,\widehat{C}) := \mathrm{N}(0,\widehat{\Cb}(\bm\param)) := \widehat{M}(\cdot|\bm\param),
\end{align*}
for \emph{fixed} $\bm\param \in \pset.$
A well-known result from~\cite{Gelbrich1990} states that
\begin{equation} \label{eq:gelbrich}
 \dWas(\kappa, \widehat{\kappa})^2 = \trace(C+\widehat{C}-2(C^{1/2}\widehat{C}C^{1/2})^{1/2}).
\end{equation}

Throughout the remainder of this section, we assume that $C - \widehat{C}$ is SPSD. Note that this assumption can be written as $\widehat{C} \preceq C$, where $\preceq$ denotes the \emph{L\"owner order} on $\mathrm{SPSD}(n)$.
If this assumption holds,  we can bound the Wasserstein distance $W_2$ between the measures $\kappa$ and $\widehat{\kappa}$ in terms of the trace of $C- \widehat{C}$.
\begin{lemma} \label{lemma_Wasserstein}
Let ${\kappa}:= \mathrm{N}(0, {C})$ and $\widehat{\kappa}:= \mathrm{N}(0, \widehat{C})$ be Gaussian measures on $\mathbb{R}^n$ and let $C- \widehat{C}$ be SPSD. Then, 
$
\dWas(\kappa, \widehat{\kappa}) \leq \trace(C - \widehat{C})^{1/2}.
$
\end{lemma}
\begin{proof}
The set of couplings $\mathrm{Coup}(\kappa,\widehat{\kappa})$ contains the following set of Gaussian measures on $\R^{2n}$:
$$E:=\left\lbrace \mathrm{N}\left(0,  G \right) : G = \begin{pmatrix}
C & \Sigma \\
\Sigma^T & \widehat{C}
\end{pmatrix},  \Sigma \in \mathbb{R}^{n \times n}, G \in \mathrm{SPSD}(2n) \right\rbrace  \subseteq \mathrm{Coup}(\kappa,\widehat{\kappa}).$$
Let $H' \in E$. Then, by definition of $\dWas$,
\begin{align*}
\dWas(\kappa, \widehat{\kappa})^2 &\leq  \int \|X_1 - X_2 \|_2^2 H'(\mathrm{d}X_1, \mathrm{d}X_2) \\
&= \int \sum_{i=1}^n  \left(X_1^{(i)} - X_2^{(i)}\right)^2  H'(\mathrm{d}X_1, \mathrm{d}X_2) \\
&=  \int \sum_{i=1}^n  \left(X_1^{(i)}\right)^2 + \left(X_2^{(i)}\right)^2 - 2\left(X_1^{(i)}X_2^{(i)}\right) H'(\mathrm{d}X_1, \mathrm{d}X_2) \\
&=  \sum_{i=1}^n  C_{ii} + \widehat{C}_{ii} - 2\Sigma_{ii} = \trace(C + \widehat{C} - 2\Sigma),
\end{align*}
where $X_j = (X_j^{(1)},\ldots, X_j^{(n)})$, $j = 1,2$.
The result of the lemma now follows from choosing $\Sigma=\widehat{C}$, provided that this choice leads to a measure in $E$, that is, the matrix
$$
G = \begin{pmatrix}
C & \widehat{C} \\
\widehat{C} & \widehat{C}
\end{pmatrix}
$$
is SPSD. On the other hand, this follows immediately from the fact that $G$ can be written as a sum of two SPSD matrices:
\[
 G(\widehat C) = \begin{pmatrix}
\widehat C^{1/2} \\
\widehat C^{1/2}
\end{pmatrix}\begin{pmatrix}
\widehat C^{1/2} \\
\widehat C^{1/2}
\end{pmatrix}^T + \begin{pmatrix}
C-\widehat{C} & 0 \\
0 & 0
\end{pmatrix},
\]
where $C-\widehat{C}$ is SPSD by assumption.
%
%
\end{proof}
Using the exact formula~\eqref{eq:gelbrich} for the Wasserstein distance, the following example verifies the tightness of the bound in \Cref{lemma_Wasserstein}.
\begin{example} \label{ex:Wassersteinbounds}
Let $n = 100$. In two settings, we consider a measure $\kappa := \mathrm{N}(0, C)$ and a sequence of measures $\widehat{\kappa}_i = \mathrm{N}(0, \widehat{C}^i)$, $i = 1,...,100$. 
\begin{itemize}
\item[(i)] We consider $C := \Id_{100}$ to be the identity matrix in $\mathbb{R}^{100 \times 100}$ and $\widehat{C}^i$ to be the diagonal matrix with $\widehat{C}^i_{jj} = 1$, if $j \leq i$, and $\widehat{C}^i_{jj}=0$, otherwise.
\item[(ii)] We consider $C := \left(\exp(-(j-k)^2)\right)_{1\leq j,k \leq 100}$  a discretization of the 1D Gaussian covariance kernel and $\widehat{C}^i := (i/100)\cdot C$.
\end{itemize}
\Cref{fig_Wassersteinbound} shows $\dWas(\kappa,\widehat{\kappa}_i)$ and its upper bound in \Cref{lemma_Wasserstein}.
\end{example}
\begin{figure}
\centering
\includegraphics[scale=0.7]{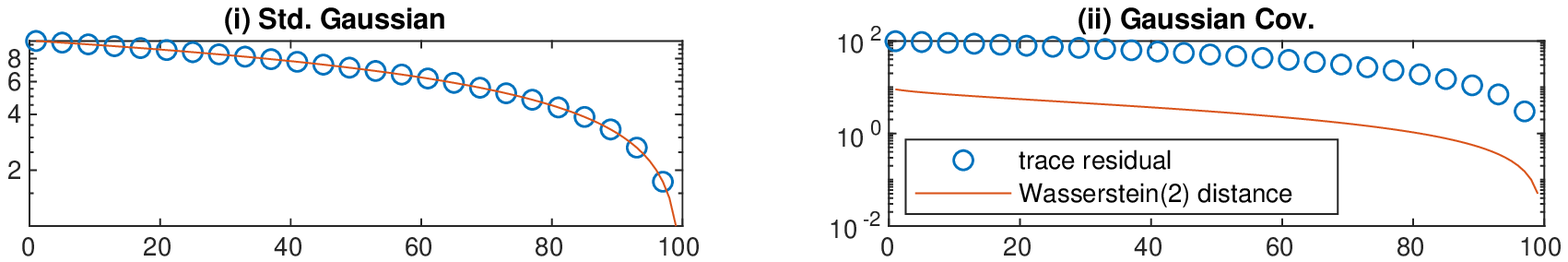}
\caption{Wasserstein distance $\dWas(\kappa,\widehat{\kappa}_i)$ and trace residual $\trace(C-\widehat{C}^i)^{1/2}$ vs. $i$ for the Gaussian measures considered in  \Cref{ex:Wassersteinbounds} (i)-(ii).}
\label{fig_Wassersteinbound}
\end{figure}
We observe that for \Cref{ex:Wassersteinbounds}(i), the bound from \Cref{lemma_Wasserstein} coincides with the Wasserstein distance. 
This is not a coincidence, since one can show the following corollary.
\begin{corollary}\label{cor_Wassersteinbound_accur}
Let ${\kappa}:= \mathrm{N}(0, {C})$ and $\widehat{\kappa}:= \mathrm{N}(0, \widehat{C})$, where $C- \widehat{C}$ is SPSD. Additionally, suppose that $C, \widehat{C}$ commute and $C^{1/2} \widehat{C}^{1/2} = \widehat{C}$. Then,
$
\dWas(\kappa, \widehat{\kappa}) =  \trace(C - \widehat{C})^{1/2}.
$
\end{corollary}
\begin{proof}
Since, $C$ and $\widehat{C}$ commute, their square-roots commute as well, and hence it follows from~\eqref{eq:gelbrich} that
\begin{align*}
\trace(C+\widehat{C}-2(C^{1/2}\widehat{C}C^{1/2})^{1/2}) &= \trace(C+\widehat{C}-2(C^{1/2}\widehat{C}^{1/2}C^{1/2}\widehat{C}^{1/2})^{1/2}) \\ &= \trace(C - \widehat{C}),
\end{align*}
using $C^{1/2} \widehat{C}^{1/2} = \widehat{C}$. Taking the square-root on both sides proves our claim.
\end{proof}
\begin{remark}\label{remark_KL_Was}
The assumptions of \Cref{cor_Wassersteinbound_accur} are for instance satisfied if $\widehat{C}$ is given via a truncation of the spectral decomposition of $C$. Such a low-rank approximation is given when employing a principal component analysis (PCA) or Karhunen-Lo\`eve (KL) expansion of the Gaussian random vectors with covariance $C$. In particular, let $V\mathrm{diag}(e_1,...,e_n) V^{-1} = C$ be the spectral decomposition of $C$ and $\widehat{C} := V\mathrm{diag}(\widehat{e}_1,...,\widehat{e}_n) V^{-1}$, where some of the eigenvalues of $C$ are replaced by zeros, i.e. $\widehat{e}_i \in \{e_i, 0\}$, $i = 1,...,n$.
The  Wasserstein distance between mean-zero Gaussian measures with these two covariance matrices is then given by $$W_2\left(\mathrm{N}(0,C), \mathrm{N}(0,\widehat{C})\right) := \left(\sum_{i=1}^n \mathbf{1}[\widehat{e}_i = 0] e_i\right)^{1/2}.$$ Notably, this is identical to the $L^2$ error obtained from dropping terms in the PCA.
\end{remark}

For general low-rank approximations $\widehat{C}$ of $C$, we cannot expect the assumptions of  \Cref{cor_Wassersteinbound_accur} to hold. Clearly, in \Cref{ex:Wassersteinbounds} (ii), the bound is not exact. However, it shows the same qualitative convergence behavior.


\subsection{The random-parameter case}
%
We now aim to generalize  \Cref{lemma_Wasserstein} to the hierarchical measures $\mu := \mu'\mathrm{N}(0,\Cb)$ and $\widehat{\mu} := \mu'\mathrm{N}(0, \widehat{\Cb})$.

\begin{theorem} \label{theorem_Wasserstein_param}
Let $\mu'$ be a probability measure on $\pset$, ${\mu}:= \mu'\mathrm{N}(0, {\Cb})$ and $\widehat{\mu}:= \mu'\mathrm{N}(0, \widehat{\Cb})$, where $\Cb(\bm\param)- \widehat{\Cb}(\bm\param)$ is SPSD for $\mu'$-a.e. $\bm\param \in \pset$. Moreover, we assume that $\int \|X\|^2_2\mu(\mathrm{d}X) < \infty.$ 
Then, 
\begin{equation}
\dWas(\mu, \widehat{\mu}) \leq \left(\int_{\pset} \trace(\Cb(\bm\param) - \widehat{\Cb}(\bm\param))\mu'(\mathrm{d}\bm\param)\right)^{1/2}. \label{eq_Wass_param}
\end{equation}
\end{theorem}

\begin{proof}
We proceed similarly as in the proof of  \Cref{lemma_Wasserstein}.
First, we define $E$ to be the set
\begin{align*}
\left\lbrace \mu'\mathrm{N}\left(0,  \Gb(\cdot) \right) : \Gb(\bm\param) = \begin{pmatrix}
\Cb(\bm\param) & \bSigma(\bm\param) \\
\bSigma(\bm\param)^T & \widehat{\Cb}(\bm\param)
\end{pmatrix} \in \mathrm{SPSD}(2n), \right. \\ \qquad \qquad \qquad \qquad  \left. \vphantom{\begin{pmatrix}
\Cb(\bm\param) & \bSigma(\bm\param) \\
\bSigma(\bm\param)^T & \widehat{\Cb}(\bm\param)
\end{pmatrix} } \bSigma(\bm\param) \in \mathbb{R}^{n \times n} \ ( \bm\param \in \pset, \mu'\text{-a.s.})  \right\rbrace.
\end{align*}
One can show that $E \subseteq \mathrm{Coup}(\mu,\widehat{\mu})$.
Let $H' \in E$. Then,
\begin{align}
\dWas(\mu, \widehat{\mu})^2 &\leq  \int \|X_1 - X_2 \|_2^2 H'(\mathrm{d}X_1, \mathrm{d}X_2) \nonumber = \int \sum_{i=1}^n  \left(X_1^{(i)} - X_2^{(i)}\right)^2  H'(\mathrm{d}X_1, \mathrm{d}X_2)  \nonumber\\
&=  \int \sum_{i=1}^n  \left(X_1^{(i)}\right)^2 + \left(X_2^{(i)}\right)^2 - 2\left(X_1^{(i)}X_2^{(i)}\right) H'(\mathrm{d}X_1, \mathrm{d}X_2). \label{eq:proofthmwasserstein}
\end{align}
By assumption, 
$\int \|X\|^2_2\mu(\mathrm{d}X) < \infty.$ Therefore, 
\begin{align*}
&\int \left(X^{(i)}\right)^2 \mu(\mathrm{d}X) < \infty,\quad \int \left(X^{(i)}\right)^2 \widehat{\mu}(\mathrm{d}X)<\infty, \\ &\Big|\int \left(X_1^{(i)}X_2^{(i)}\right) H'(\mathrm{d}X_1, \mathrm{d}X_2)\Big| < \infty,
\end{align*}
where the second moment of $\widehat{\mu}$ is finite, since $\Cb(\bm\param) - \widehat{\Cb}(\bm\param)$ is almost surely SPSD.
Hence, the integral in \cref{eq:proofthmwasserstein} is finite, and 
\begin{align*}
\text{\eqref{eq:proofthmwasserstein}} &=  \iint \sum_{i=1}^n  \left(X_1^{(i)}\right)^2 + \left(X_2^{(i)}\right)^2 - 2\left(X_1^{(i)}X_2^{(i)}\right) \mathrm{N}(0,\Gb(\bm\param))(\mathrm{d}X_1, \mathrm{d}X_2)\mu'(\mathrm{d}\bm\param) \\ &=  \int\sum_{i=1}^n  \Cb(\bm\param)_{ii} + \widehat{\Cb}(\bm\param)_{ii} - 2\bSigma(\bm\param)_{ii} \mu'(\mathrm{d}\bm\param) \\
&= \int \trace(\Cb(\bm\param) + \widehat{\Cb}(\bm\param) - 2\bSigma(\bm\param))\mu'(\mathrm{d}\bm\param),
\end{align*}
where $\bSigma(\bm\param)$ is chosen such that $\Gb(\bm\param)$ is SPSD, $\bm\param \in \pset$, $\mu'$-a.s.
Analogously to the proof of \Cref{lemma_Wasserstein}, by setting $\bSigma(\bm\param) \equiv \widehat{\Cb}(\bm\param)$ we have  that $\mu'$-a.s. $\Gb(\bm\param)$ is SPSD.
This yields the upper bound claimed in the theorem.
\end{proof}
In practice, we cannot compute the integral on the right-hand side of \cref{eq_Wass_param}.
A straight-forward alternative is a Monte Carlo approximation with samples from $\mu'$. We derive a Monte-Carlo-based bound in the corollary below and quantify its reliability.
\begin{corollary} \label{cor:MonteCarlo}
Let  $\bm\param_1,...,\bm\param_M   \sim \mu'$ be independent and identically distributed (i.i.d.)  and $\mathrm{Var}\left( \trace(\Cb(\bm\param_1) - \widehat{\Cb}(\bm\param_1)) \right) < \infty$. Then, for all $\varepsilon > 0$, we have
$$\dWas(\mu, \widehat{\mu})^2  \leq \frac{1}{M}\sum_{m=1}^M \trace(\Cb(\bm\param_m) - \widehat{\Cb}(\bm\param_m)) + \varepsilon$$
with probability  not less than $1 - \mathrm{Var} (\trace(\Cb(\bm\param_1) - \widehat{\Cb}(\bm\param_1)))/{(M \varepsilon^2)}.$
\end{corollary}
\begin{proof}
Let $\bm\param_1,...,\bm\param_M$ be random variables  defined on $(\Omega, \mathcal{F}, \mathbb{P})$ taking values in $(\pset, \mathcal{A})$ and be distributed as stated above. 
Moreover, we define $\bm\xi_m :=  \trace(\Cb(\bm\param_m) - \widehat{\Cb}(\bm\param_m))$. for $m = 1,\ldots,M$.

By \Cref{theorem_Wasserstein_param}, we have $\dWas(\mu, \widehat{\mu})^2 \leq \mathbb{E}[\bm\xi_1]$. 
Hence,
\begin{align*}
\mathbb{P}\left(\dWas(\mu, \widehat{\mu})^2  \leq \frac{1}{M}\sum_{m=1}^M \bm\xi_m + \varepsilon \right) 
&\geq \mathbb{P}\left(\mathbb{E}[\bm\xi_1]  \leq \frac{1}{M}\sum_{m=1}^M \bm\xi_m + \varepsilon \right) \notag \\
&\stackrel{\varepsilon > 0}{\geq} \mathbb{P}\left( \Large| \mathbb{E}[\bm\xi_1]  - \frac{1}{M}\sum_{m=1}^M \bm\xi_m \Large|  \leq \varepsilon \right)  \\ &\geq 1- \frac{\mathrm{Var}(\frac{1}{M}\sum_{m=1}^M \bm\xi_m)}{\varepsilon^2} \stackrel{\text{i.i.d.}}{=} 1- \frac{\mathrm{Var}(\xi_1)}{M \cdot  \varepsilon^2},
\end{align*}
where we employ the Chebyshev inequality in the second to last step.
\end{proof}

\section{Low-rank approximation of the covariance operator}\label{Sec_aca}
\subsection{Cross approximation}
%
%

Cross approximation \cite{bebendorf} is a technique to construct a low-rank approximation of a matrix from some of its rows and columns. More specifically, 
for a matrix $A\in\mathbb R^{n\times n}$ let us consider the rows $A(I,:)$ and columns $A(:,J)$ corresponding to 
index sets
$I=\{i_1,\dots,i_k\}$ and $J=\{j_1,\dots,j_k\}$, respectively. Assuming that the cross matrix
\[
A(I,J) = \begin{bmatrix}
          a_{i_1,j_1} & a_{i_1,j_2} & \cdots & a_{i_1,j_k} \\
          a_{i_2,j_1} & a_{i_2,j_2} & \cdots & a_{i_2,j_k} \\
\vdots & \vdots & & \vdots \\
          a_{i_k,j_1} & a_{i_k,j_2} & \cdots & a_{i_k,j_k} 
          \end{bmatrix}
\in \R^{k\times k}
\]
is invertible, the cross approximation of $A$ associated with $I$ and $J$ is given by \begin{equation}\label{eq:cross} A(:, J)A(I,J)^{-1}A(I, :).\end{equation}
This approximation has rank $k=|I|=|J|$.


\subsection{Adaptive cross approximation for an SPSD matrix}

The choices of $k$ and of the index sets $I,J$ are crucial for the accuracy of the cross approximation~\eqref{eq:cross}.
The \emph{adaptive cross approximation} algorithm (ACA) chooses $I,J$ in an adaptive way via a greedy strategy, analogous to complete pivoting in Gaussian elimination. In the first step of the procedure, $i_1,j_1$ are chosen such that the pivot element $A_{i_1,j_1}$ has maximal absolute value. Then the procedure is repeated for $A$ replaced by the residual matrix $A-A(:,i_1) A(i_1,j_1)^{-1} A(:,j_1)^T$, and so on. For an SPSD matrix $A$, the entry of maximal absolute value can always be found on the diagonal. As the residual matrix is again SPSD, this implies that ACA for SPSD matrices chooses index sets $I,J$ such that $I=J$. This greatly reduces the cost of ACA because only the diagonal elements of the residual matrices need to be inspected in order to determine the pivot elements.
Therefore, ACA has linear complexity with respect to $n$ and returns an SPSD cross approximation of the form
\[
 A_I := A(:,I) A(I,I)^{-1} A(:,I)^T.
\]
As the residual matrix $E_I:=A-A_I$ is also SPSD, its trace equals its nuclear norm:
\begin{equation} \label{eq:errortrace}
\norm{E_I}_*:=\sum_{j=1}^n\sigma_j(E_I)=\trace(E_I)=
\trace(A - A(:,I) A(I,I)^{-1} A(:,I)^T),
\end{equation}
with $\sigma_{j}(\cdot)$ denoting the $j$th singular value.
In view of \Cref{lemma_Wasserstein}, if $A$ is a covariance matrix then $\trace(E_I)^{1/2}$ is an upper bound for the squared Wasserstein distance $W_2$ between the Gaussian measures $\mathrm{N}(0, A)$ and $\mathrm{N}(0, A_I)$. The a priori error bound $\trace(E_I)\leq 4^k(n-k)\sigma_{k+1}(A)$, where $k=|I|$, has been proved in \cite[Theorem 1]{foster}. This estimate is pessimistic; typically ACA yields an approximation error that remains much closer to $\sigma_{k+1}(A)$. 

The adaptive cross approximation algorithm for SPSD matrices \cite{harbrecht} is reported in \Cref{alg:aca}. To simplify the description, $E_I$ is initialized to the complete matrix $A$. In a practical implementation, only the diagonal elements of $A$ are used and line~\ref{step:residual} is replaced by updating the diagonal entries of $E_I$ only. In turn, only the diagonal entries and the columns corresponding to $I$ of $A$ need to be evaluated. 
The cost of $k$ iterations of \Cref{alg:aca} is $\mathcal O( (k+c_A) k n)$, where $c_A$ denotes the cost of evaluating an entry of $A$.

\begin{algorithm} \small 
	\caption{Adaptive cross approximation for an SPSD matrix}\label{alg:aca}
\begin{algorithmic}[1]
					\Procedure{ACA}{$A,\mathsf{tol}$, $k_{\mathsf{max}}$}
\State Set $E_I:=A$, $k:=0$, $I:= \emptyset$ 
\For{$k:=1,2,\dots, k_{\mathsf{max}} $} 
    \State $i_k := \arg\max_{i} (E_{I})_{ii} $  
    \State $I \gets I\cup\{i_k\}$ 
    \State $u_k := E_I(:,i_k) / \sqrt{(E_I)_{i_ki_k}}$  
    \State $E_I := E_I - u_k u_k^T$ \label{step:residual}
	    \If {$\trace(E_I) \le \mathsf{tol}$} \label{step:stop} 
		\State \textbf{break}
		\EndIf
\EndFor
\State \Return $I$
\EndProcedure
\end{algorithmic}
\end{algorithm}

\subsection{Adaptive cross approximation for a parameter-dependent SPSD matrix} \label{subsec_ACA_param}
The approximation of a parameterized Gaussian random field requires the approximation of a parameterized family of covariance matrices, see \Cref{theorem_Wasserstein_param}. In this section we introduce an extension of ACA to the case of a parameterized family of matrices $A(\bm\param): \pset \to \R^{n\times n}$ such that $A(\bm\param)$  has an affine linear expansion:
\begin{equation} \label{eq:affinelinear}
A(\bm\param) = \varphi_1(\bm\param) A_1 + \cdots + \varphi_s(\bm\param)A_s , \quad \varphi_j:\pset \to \R, \quad A_j \in \R^{n\times n}, \qquad j=1,\dots,s
\end{equation}
and $A(\bm\param)$ is SPSD for each $\bm\param \in \pset$.  Again, we aim at designing a matrix free method, so we assume to have an efficient way to extract entries from the matrices $A_1,\dots,A_s$ in place of forming them explicitly. Moreover, we replace the parameter space $\pset$ with a finite surrogate set $\pset_f=\{\bm\param_1,\dots,\bm\param_m\} \subset \pset$, e.g. a uniform sampling over $\pset$ or, in the light of \Cref{cor:MonteCarlo}, samples from $\mu'$.

The goal of the algorithm is to provide a subset of indices $I$ such that \begin{equation} \label{eq:defAI}
A_I(\bm\param):=A(\bm\param)(:,I)[A(\bm\param)(I,I)]^{-1}A(\bm\param)(:,I)^T                                                                           
                                                                          \end{equation}
is a uniformly good approximation of $A(\bm\param)$, i.e. $\trace(A(\bm\param)-A_I(\bm\param))\leq \mathsf{tol}$ for any $\bm\param\in\pset_f$. Following a popular technique in reduced basis methods~\cite{Hesthaven2015}, the core idea of the method is to augment the set $I$ with a greedy strategy. First, the value $\bm\param^*$ which verifies  \begin{equation}\label{eq:max-trace}\trace(A(\bm\param^*)-A_I(\bm\param^*))=\max_{\bm\param\in\pset_f}\trace(A(\bm\param)-A_I(\bm\param)),\end{equation} is identified. Then, one step of \Cref{alg:aca} is applied to the matrix $A(\bm\param^*)-A_I(\bm\param^*)$ returning the new index that is appended to $I$. 

The most expensive part of the method sketched above is finding $\bm\theta^*$ since it requires to evaluate the objective function  for any point in $\pset_f$. In order to amortize the cost of multiple evaluations of $\trace(A(\bm\param)-A_I(\bm\param))=\trace(A(\bm\param)) - \trace( A_I(\bm\param))$ we precompute once and for all 
\[
t_j:= \trace(A_j),  \quad j= 1,\ldots,s
\]
so that for every $\bm\param \in \pset_f$, computing the quantity
\[
\trace(A(\bm\param))=\varphi_1(\bm\param)t_1+\dots+\varphi_s(\bm\param)t_s
\]
requires $2s -1$ flops and $s$ evaluations of the scalar functions $\varphi_j$.

In addition, whenever $I$ is updated, we store the matrices $A_1(I, I),\dots,A_s(I, I)$ and we compute the compact  QR factorization $Q_I\in\mathbb R^{n\times sk}$, $R_I\in\mathbb R^{sk\times sk}$ of  $[A_1(:, I),\dots,A_s(:, I)]$. In particular, for all $\bm\param\in\pset$ we have that 
\[
A(\bm\param)(:,I)= Q_IR_I\mathbf{\varphi}(\bm\param),\qquad \mathbf{\varphi}(\bm\param):=\begin{bmatrix}\varphi_1(\bm\param)\Id_{k}\\\vdots\\\varphi_s(\bm\param)\Id_{k}\end{bmatrix}.
\]
By denoting with $R_A(\bm\param)$ the Cholesky factor of $A(\bm\param)(I,I)$,  we can rewrite $\trace( A_I(\bm\param))$ as follows:
\begin{align}\label{eq:trace}
	\trace(A_I(\bm\param))&= \trace( A(\bm\param)(:,I) [A(\bm\param)(I,I)]^{-1} A(\bm\param)(:,I)^T  ) \nonumber\\
	&= \trace( Q_IR_I\mathbf{\varphi}(\bm\param) R_A(\bm\param)^{-1}R_A(\bm\param)^{-T}\mathbf{\varphi}(\bm\param)^TR_I^TQ_I^T) \nonumber\\
	&=  \trace( R_A(\bm\param)^{-T}\mathbf{\varphi}(\bm\param)^TR_I^TR_I\mathbf{\varphi}(\bm\param)R_A(\bm\param)^{-1})\\
	&=\norm{R_I\mathbf{\varphi}(\bm\param)R_A(\bm\param)^{-1}}_F^2,\nonumber
\end{align}
where  we used the cyclic property of the trace for deriving the third equality. 

The considerations above lead to \Cref{alg:param-aca}.
\begin{remark} \em 
	Instead of using the QR decomposition, the computation of $\trace(A_I(\bm\param))$ can also be carried out by precomputing the quantities 
	\[
	B_{ij}:=A_i(:,I)^TA_j(:,I),\qquad i,j=1,\dots,s
	\] 
	and exploiting the relation
\begin{align*}
\trace(A_I(\bm\param))&= \trace(R_A(\bm\param)^{-T}A(\bm\param)(:,I)^T   A(\bm\param)(:,I) R_A(\bm\param)^{-1} ) \nonumber\\
&=  
\trace(R_A(\bm\param)^{-T}B(\bm\param)R_A(\bm\param)^{-1}),
\end{align*}
where $B(\bm\param):=\sum_{i,j=1}^s\varphi_i(\bm\param)\varphi_j(\bm\param)B_{ij}$.
This formula might be implemented in place of \eqref{eq:trace} which requires to compute the QR decomposition of $A(\bm\param)(:,I)$. However, forming the matrix $B(\bm\param)$, which is the principal submatrix of $A(\bm\param)^2$ corresponding to the index set $I$, causes a loss of accuracy in the evaluation of the stopping criterion. Due to the squaring of the singular values of $A(\bm\param)(:,I)$, the computed quantity is unreliable when  $\trace(E_I)$ is below the square root of the machine precision. 
\end{remark}
\begin{algorithm} \small 
	\caption{Adaptive cross approximation for a parameterized family of SPSD matrices}\label{alg:param-aca}
	\begin{algorithmic}[1]
		\Procedure{param\_ACA}{$A_1,\dots,A_s,\varphi_1,\dots,\varphi_s,\pset_f,\mathsf{tol}$}
		\State Set $I:= \emptyset$ 
		\For{$j:=1,\dots,s$} 
		\State $t_j := \trace(A_j) $  
		\EndFor
		\For{$k:=1,2,\dots$} 
		\State Store matrices $A_1(I,I),\dots,A_s(I,I)$
		\State Compute $R_I$ such that $Q_IR_I=[A_1(:,I),\dots,A_s(:,I)]$ \label{line:updateRI}
		\State Set $\mathsf{res}_{\mathsf{max}}:=0$
		\For{$\bm\param:=\bm\param_1,\dots,\bm\param_m$}\label{step:inner-loop}
		\State Compute  $A(\bm\param)(I,I)=\sum_{j=1}^s\varphi_j(\bm\param)A_j(I,I)$
		\State	 Compute $R_A(\bm\param)$ such that $R_A(\bm\param)^TR_A(\bm\param)=A(\bm\param)(I,I)$
		\State $\mathsf{res}= \sum_{j=1}^s\varphi_j(\bm\param)t_j-\norm{R_I\mathbf{\varphi}(\bm\param)R_A(\bm\param)^{-1}}_F^2$   
\If {$\mathsf{res} \geq \mathsf{res}_{\mathsf{max}}$} 
\State $\mathsf{res}_{\mathsf{max}}=\mathsf{res}$ and $\bm\param^*=\bm\param$
\EndIf
		\EndFor
		\If {$\mathsf{res}_{\mathsf{max}} \le \mathsf{tol}$}  \label{line:stoppara}
		\State \textbf{break}
		\EndIf
		\State $i_{\mathsf{new}}=\Call{ACA}{A(\bm\param^*)-A_I(\bm\param^*), 0, 1}$
				\State $I \gets I\cup\{i_{\mathsf{new}}\}$
		\EndFor
		\State \Return $I$
		\EndProcedure
	\end{algorithmic}
\end{algorithm}

Note that the matrix $Q_I$ in line~\ref{line:updateRI} is actually not needed. More importantly, a careful implementation of \Cref{alg:param-aca} recognizes that each $A_j(:,I)$ grows by a column at a time. Using QR updating techniques~\cite{Golub2013}, the cost of executing line~\ref{line:updateRI} in each outer loop is 
$\mathcal O(ns^2k)$ --- instead of $\mathcal O(ns^2k^2)$ when computing the QR decomposition from scratch in each loop. Similarly, updating techniques could be used to accelerate the inner loop, but these are likely less important as long as $s$ and $k$ are much smaller than $n$.
Without such techniques, one inner loop costs $\mathcal O(sk^3+s^2k^2)$. In summary, the total cost of executing $k$ iterations of \Cref{alg:param-aca}
is $\mathcal O(ns^2k^2+m(sk^4+s^2k^3))$.
Additionally, $\mathcal O( kn s)$ evaluations of entries from the coefficients $A_j$ are needed.

\subsection{Robustness of ACA}

As we will explain in \Cref{Sec_IsotrCKer} the affine linear expansion~\eqref{eq:affinelinear} is often assured by an approximation of the true covariance operator.
This approximation error may destroy positive semi-definitness, a property that is assumed throughout the derivations above. In this case, we propose to use the affine linear expansion only for identifying the index set $I$, i.e. as input of Algorithm~\ref{alg:param-aca}, and then use the true covariance operator to form the cross approximation in the online phase. This last step is essential to ensure that the residual matrix is still SPSD and its trace represents a bound for the Wasserstein distance with respect to the true covariance operator.

We remark that running Algorithm~\ref{alg:aca} on a symmetric indefinite matrix $A$ returns an SPSD cross approximation $A_I$, because the method always chooses positive pivots.     
However, the stopping criterion at line~\ref{step:stop} of \Cref{alg:aca} loses its justification because the relation~\eqref{eq:errortrace} between the trace and the nuclear norm does not hold. On the other hand, if $A$ is very close to an SPSD matrix $\widetilde A$, then it is likely that $I$ is a good choice for the cross approximation of $\widetilde A$.  The next lemma establishes a result in this direction. The argument used in proof is inspired by the stability analysis of the pivoted Cholesky factorization~\cite[Section 10.3.1]{higham}. 
\begin{lemma} \label{Lemma_approxim}
	Consider symmetric matrices $A,E \in\mathbb R^{n\times n}$ such that $\widetilde A= A-E$ is SPSD and $\norm{E}_2\le \delta$ for some $\delta > 0$. Let $I$ be the index set returned  by \Cref{alg:aca} applied to $A$ with tolerance $\mathsf{tol}$. If $\rho:=\delta \norm{A(I,I)^{-1}}_2<1$ then 
	\[
	\trace (\widetilde A- \widetilde A_I)\leq \mathsf{tol} + (n-k)\delta \frac{(1+\norm{W}_2)^2}{1-\rho}
	\]
	where $\widetilde A_I:=\widetilde A(:,I) \widetilde A(I,I)^{-1} \widetilde A(I,:)$, $W:=A(I,I)^{-1}A(I,I^c)$ and
	$I^c:=\{1,\dots,n\}\setminus I$.
	\end{lemma}
\begin{proof}
 Consider the Schur complements
\[
 S=A(I^c,I^c)-A(I^c,I) A(I,I)^{-1} A(I,I^c), \qquad 
 \widetilde S=\widetilde A(I^c,I^c)-\widetilde A(I^c,I) \widetilde A(I,I)^{-1} \widetilde A(I,I^c).
\]
Because $\rho < 1$, the Neumann series gives $$\widetilde A(I,I)^{-1}=A(I,I)^{-1}+ \sum_{j>0}(A(I,I)^{-1}E(I,I))^jA(I,I)^{-1}.$$  
This yields
\begin{align*}
\widetilde S&= S+E(I,I) + W^TE(I,I)\sum_{j\geq 0}[A(I,I)^{-1}E(I,I)]^jW \\ &+ E(I^c,I)\sum_{j\geq 0}[A(I,I)^{-1}E(I,I)]^jW+W^T\sum_{j\geq 0}[A(I,I)^{-1}E(I,I)]^jE(I, I^c)\\
&+E(I^c, I)\sum_{j\geq 0}[A(I,I)^{-1}E(I,I)]^jA(I,I)^{-1}E(I, I^c),
\end{align*}
 see also \cite[Lemma 10.10]{higham}. In turn, 
$\norm{S- \widetilde S}_2 \leq \delta\left(1 + \frac{\norm{W}_2^2+2\norm{W}_2+\rho}{1-\rho}\right)=\delta\frac{(1+\norm{W}_2)^2}{1-\rho}$.
Because $S$ and $\widetilde S$ are symmetric matrices, standard eigenvalue perturbation results~\cite[Thm 8.1.5]{Golub2013} imply that
\[ |\lambda_j-\widetilde\lambda_j|\leq \delta\frac{(1+\norm{W}_2)^2}{1-\rho},\] where 
$\lambda_j$ and $ \widetilde \lambda_j$, $j=1,\dots,n-k$, denote the ordered eigenvalues of $S$ and $\widetilde S$, respectively.
This bound implies 
\begin{align*}
\norm{\widetilde S}_* &=  \sum_{j = 1}^{n-k} \widetilde \lambda_j \le  \trace( S) +  \sum_{j = 1}^{n-k} |\lambda_j- \widetilde \lambda_j|\le \mathsf{tol}+(n-k)\delta\frac{(1+\norm{W}_2)^2}{1-\rho},
\end{align*}
where the last inequality uses the stopping criterion $\mathsf{tol}\geq\trace( A- A_I)=\trace( S)$ of \Cref{alg:aca}.
	\end{proof}
	
 The upper bound in \Cref{Lemma_approxim} indicates that the impact of the inexactness $\delta$ on the robustness of the trace criterion \eqref{eq:errortrace} may be magnified by the quantity $\norm{W}_2^2$. Typically, $\norm{W}_2^2$ remains moderate but it is difficult to provide a rigourous meaningful bound. The quite pessimistic bound $\norm{W}_F^2\leq \frac 13(n-k)(4^{k}-1)$ applies to some particular cases, see \cite[Lemma 10.13]{higham}. 
 
  An analogous result can be stated for the parameter dependent case where $\widetilde{A}(\bm\param)$ represents the full covariance operator, $A(\bm\param)$ is its linear separable approximation, and $E(\bm\param)=A(\bm\param)-\widetilde{A}(\bm\param)$ represents the approximation error. Since this generalization of \Cref{Lemma_approxim} is rather straightforward we refrain from reporting it.

\subsection{ACA sampling and its complexity}\label{sec:sampling}
Let us now consider a parameterized Gaussian measure $M$, as introduced in \Cref{Subs_Problem}, with linearly separable covariance operator $\Cb: \pset \rightarrow \mathrm{SPSD}(n)$. In the following, we explain how the ACA approximation of $\Cb$ enables us to sample fast from an approximation of $M(\cdot|\bm\param)$ for a large number of $\bm\param \in \pset$.

After applying \textsc{param\_ACA} (\Cref{alg:param-aca}) to $\Cb$, we obtain the approximate parameterized matrix $\Cb_I: \pset \rightarrow \mathrm{SPSD}(n)$ defined in~\eqref{eq:defAI}. The corresponding approximate Markov kernel $\widehat{M}$ is given by
$
\widehat{M}(\cdot |\bm\param) := \mathrm{N}(0, \Cb_I(\bm\param)). 
$
After this \emph{offline phase}, we proceed with the \emph{online phase}, that is, the sampling from $\widehat{M}$. 
For a given parameter $\bm\param \in \pset$, we compute the Cholesky factor $L \in \mathbb{R}^{k \times k}$ of $\Cb(\bm\param)(I,I)$.
For a sample $\xi \sim \mathrm{N}(0, \mathrm{Id}_{k})$ we have \begin{equation} \label{eq:approxsample}
\Cb(\bm\param)(:,I)L^{-T}\xi \sim \widehat{M}(\cdot | \bm\param) = \mathrm{N}(0, \Cb_I(\bm\param))                                                                
                                                               \end{equation}
because \begin{align*}
\Cb(\bm\param)(:,I)L^{-T}(\Cb(\bm\param)(:,I)L^{-T})^T &= \Cb(\bm\param)(:,I)L^{-T}L^{-1}\Cb(\bm\param)(:,I)^T \\&= \Cb(\bm\param)(:,I) \Cb(\bm\param)(I,I)^{-1}\Cb(\bm\param)(:,I)^T= \Cb_I(\bm\param).
\end{align*}
We now discuss the computational cost of drawing a sample from $\widehat{M}(\cdot | \bm\param)$ via~\eqref{eq:approxsample}. 
First, we need to construct the matrices $\Cb(\bm\param)(:,I), \Cb(\bm\param)(I,I)$ which comes at a cost of $\mathcal{O}(nks)$ if the linear expansion of $\Cb$ has $s$ terms. Then, we compute the Cholesky factor $L$ of $\Cb(\bm\param)(I,I)$ with $\mathcal{O}(k^3)$ operations, sample $\xi \sim  \mathrm{N}(0, \mathrm{Id}_{k})$,  and solve $L^T b = \xi$ at a cost of $\mathcal{O}(k^2)$. Finally, the sample is obtained by the matrix vector product $\Cb(\bm\param)(:,I) b$, which costs another $\mathcal{O}(nk)$.
In summary, the computational cost of one sample in the online phase is given by
$
\mathcal{O}(kns + k^3).
$
Hence, the cost of the complete procedure for drawing $u \in \mathbb{N}$ samples is given by
$$
\mathcal{O}(\underbrace{k^2 ns^2+m(k^4s+k^3s^2)}_{\text{offline}} + \underbrace{knsu + k^3u}_{\text{online}}). 
$$
Hence, the cost is linear in the size of the matrix $n$, the number of test parameters $m$, and the number $u$ of different samples. 
As our approach is most meaningful in situations where the rank $k$ is much smaller than $n$ and $s$ is small, the terms $k^2 ns^2$ and $knsu$ will typically dominate the offline and online phases, respectively. Let us assume that applying the non-parameterized ACA $u$ times costs $\mathcal{O}((\widetilde k+s)\widetilde k n u)$, with $\widetilde k\leq k$. Then, sampling with \textsc{param\_ACA} can be expected to be cheaper as long as $\widetilde k^2> ks$ and $m$ is not too large. 

When the affine expansion \eqref{eq:affinelinear} is an approximation of the true covariance operator, we use the latter both in the online phase of the approach based on \textsc{param\_ACA} and when applying the non-parameterized \textsc{ACA}. This changes the complexity of the two approaches by replacing the parameter $s$ with the cost of evaluating one entry of the true covariance operator. 

\begin{remark}
Working directly with the true covariance operator inside \textsc{param\_ACA}, i.e., identifying $\max_{\bm\param\in\bm\pset_f}\trace(A(\bm\param)-A_I(\bm\param))$ by applying $|I|$ steps of \textsc{ACA} for each value of $\bm\param$, yields an overall cost of $\mathcal O(nmk^3)$ for the offline phase. Although this is a much simpler approach, it provides a significantly higher cost when considering fine discretizations of $\pset$.
\end{remark}

\section{Isotropic covariance kernels}  \label{Sec_IsotrCKer}
 
Let $D \subseteq \mathbb{R}^{\ell}$ be some domain and $\| \cdot \|_D$ be a norm on $D$. A {covariance kernel} is a function $c: D\times D \rightarrow \mathbb{R}$ that is continuous, symmetric, and positive semidefinite. In the following, we consider \emph{isotropic} covariance kernels $c$, which take 
$$c(\mathbf x, \mathbf y) = \tilde{c}(\|\mathbf x-\mathbf y\|_D) \qquad \forall\mathbf x,\mathbf y \in D$$
for some function $\tilde{c}: [0, \infty) \rightarrow \mathbb{R}$.
Such a kernel is invariant with respect to isometries, e.g., translations and rotations in the case of the Euclidean norm.
Given points $\mathbf x_1,\ldots,\mathbf x_n \in D$, the covariance matrix $C \in \mathbb{R}^{n \times n}$ associated with $c$ is defined by
\begin{equation} \label{eq:covmatrix}
C_{ij} = c(\mathbf x_i, \mathbf x_j) \qquad (i,j = 1,\ldots,n).
\end{equation}
Parameterized covariance kernels are functions $c: D\times D \times \pset \rightarrow \mathbb{R}$ such that $c(\cdot, \cdot; \bm\param)$ is a covariance kernel for any $\bm\param \in \pset$. The associated parameterized covariance matrix $\mathbf{C}: \pset \rightarrow \mathbb{R}^{n \times n}$ is defined analogously as in~\eqref{eq:covmatrix}.

\subsection{Examples}
The following examples of parameterized covariance kernels will be considered throughout the rest of this work. We refer to the work by Stein \cite{stein} for further examples.
\begin{paragraph}{Gaussian covariance} \label{ex:gauss}
 Let $D \subseteq \mathbb R^\ell$, $\ell\in\mathbb N$ and $\param \in \pset := [\underline{\param}, \overline{\param}]$ for $\underline{\param}>0$. 
We consider a Gaussian covariance kernel on $D$  with correlation length $\param$  defined as \begin{equation}\label{eq:gauss-ker}c(\mathbf x,\mathbf y; \param) = \sigma^2\exp\left(- \frac{ \|\mathbf x-\mathbf y \|^2_2}{2 \param^2} \right) \qquad (\mathbf x,\mathbf y \in D).\end{equation}
for some $\sigma^2 \in (0, \infty)$.
The correlation length determines the strength of correlation between two points depending on their distance.
\end{paragraph}
\begin{paragraph}{Mat\'ern covariance}  \label{ex:Matern}
 We let $D \subseteq \mathbb R^\ell$, $\ell\in\mathbb N$ and $\bm\param \in \pset := [\underline{\param}, \overline{\param}]\times (0, \infty]$ for $\underline{\param}>0$.
The Mat\'ern covariance kernel on $D$  with correlation length $\param_1$ and smoothness $\param_2$ is defined as \begin{equation}\label{eq:matern}c(\mathbf x,\mathbf y; \bm \param) = \sigma^2 \frac{2^{1-\theta_2}}{\Gamma(\theta_2)}\left(\frac{\sqrt{2\theta_2}\|\mathbf x-\mathbf y\|_2}{\theta_1} \right)^{\theta_2}K_{\theta_2}\left(\frac{\sqrt{2\theta_2}\|\mathbf x-\mathbf y\|_2}{\theta_1} \right) \qquad (\mathbf x,\mathbf y \in D),\end{equation}
where $\Gamma$ is the gamma function and $K_{\theta_2}$ is the modified Bessel function of second kind with parameter $\theta_2$.

The correlation length has the same interpretation as for Gaussian kernels. Moreover,
a Gaussian kernel with correlation length $\theta_1$ can be viewed as a special case of a Mat\'ern kernel by taking the limit $\theta_2 \rightarrow \infty$.
For finite values, the parameter $\theta_2$ determines the smoothness of the random samples; see, e.g., \cite[Proposition 3.2]{Stuart2018} for details.
\end{paragraph}

We aim at applying \Cref{alg:param-aca} to obtain low-rank approximations of the covariance matrix $\mathbf{C}(\bm\param)$ associated with any of the parameterized covariance kernels above. However, the linear separability assumption~\eqref{eq:affinelinear} is not satisfied because the kernels depend nonlinearly on $\bm\param$.
To address this issue, an approximate linearization of Mat\'ern kernels using Taylor expansion has been proposed in~\cite{LATZ2019FastSa}.
For smaller correlation lengths, this method becomes inefficient and we therefore discuss more general and more effective methods for linearising isotropic covariance kernels in the following.

\subsection{Linearization of isotropic covariance kernels}

Finding an (approximate) expansion~\eqref{eq:affinelinear} for $\mathbf{C}(\theta)$ is directly related to finding an expansion that separates 
the spatial variables from the parameters in the covariance kernel. To see this, let us 
consider a parameterized isotropic covariance kernel  $c(\mathbf x,\mathbf y; \bm\param) = \tilde{c}(\|\mathbf x-\mathbf y \|_2,\bm\param)$ with $\tilde{c}:[\underline d,\overline d]\times \pset\rightarrow  \mathbb R$,
 $\underline d=\min_D\|\mathbf x-\mathbf y\|_2$, and $\overline d=\max_D\|\mathbf x-\mathbf y\|_2$.  If 
 \begin{equation} \label{eq:separatetildec}
\tilde{c}(d,\bm\param) \approx \tilde{c}_s(d,\bm\param) := \sum_{j=1}^s \varphi_j(\bm\param)a_j(d)  
 \end{equation}
 for functions $\varphi_j:\pset\rightarrow  \mathbb R$ and $a_j : [\underline d,\overline d] \to \mathbb R$ then
\begin{equation}\label{eq:fun-aca}
c(\mathbf x,\mathbf y;\bm\param) \approx c_s(\mathbf x,\mathbf y;\bm\param) := \sum_{j=1}^s \varphi_j(\bm\param) a_j(\|\mathbf x-\mathbf y \|_2). 
\end{equation}
By defining $A_j$ to be the matrix associated with $a_j(\|\mathbf x-\mathbf y \|_2)$, as in~\eqref{eq:covmatrix}, it follows that the matrix
$\mathbf C_s(\bm\param)$ is associated with $c_s(\mathbf x,\mathbf y;\bm\param)$.
Note that, unless \eqref{eq:fun-aca} is an exact expansion, the matrices $\mathbf C_s(\bm\param)$ --- while still being symmetric --- are not guaranteed to be positive semi-definite for every $\bm\param \in \pset$. 
\Cref{Lemma_approxim} indicates that ACA remains robust despite the loss of semi-definiteness and, assuming that the tolerance of ACA is on the level of the approximation error of~\eqref{eq:separatetildec} then the total error committed by ACA and the expansion can be expected to stay on the same level.

The separable expansion~\eqref{eq:separatetildec} is the continuous analogue of low-rank matrix approximation. Several methods have been proposed to address this task, see, e.g.,~\cite{bebendorfACA,townsendthesis,nouy}. 
In the following sections we describe how we compute such an expansion depending on the  number of parameters.

\subsubsection{1D parameter space}\label{subsubsec_linearly_ACA}

 When there is only one parameter $\theta$ and $\pset$ is an interval then $\tilde c(d,\param)$ becomes a bivariate function on a rectangular domain. In this situation  we can apply the continuous analogue of the ACA for bivariate functions~\cite{bebendorfACA}. To implement this method in practice, it needs to be combined with a discretization of the variables. In~\cite{townsend13}, the Chebyshev expansion is proposed for the latter, an approach implemented in the command \texttt{cdr} of the toolbox Chebfun2, which is used in our experiments.

\begin{example} \label{exam_functi_ACA_Gauss}
Let us consider the Gaussian covariance kernel from \Cref{ex:gauss} with $D=[0,1]^2$, so that $[\underline d,\overline d]=[0,\sqrt 2]$   and  $\pset =  [0.1, \sqrt{2}]$.  The accuracy of the separable expansion \eqref{eq:fun-aca} --- computed via the Chebfun2 command \texttt{cdr}  --- is shown in \Cref{Figure_error_evolution_ACA_functions} (left);  we plot the maximum absolute error $\max_{d\in[0,\sqrt 2]}|\tilde c(d,\param) - \tilde{c}_K|$ --- as $K$ increases --- for different choices of $\param$ in $\pset$.
We remark  that --- already for $K=1,2$ --- $\tilde c_K$ is  accurate when considering $\param=\underline{\param} = 0.1$ and $\param=\overline{\param} = \sqrt{2}$ because the ACA for functions selects those values of $\param$ as pivots. 

 We repeat the experiment for a smaller minimal correlation length $\underline{\param} = 0.001$.
 This case is problematic when using polynomial expansions for performing the approximation~\eqref{eq:separatetildec}; see the results reported in \cite{LATZ2019FastSa}.  
 Although its convergence is somewhat slower, the ACA for functions still performs reasonably well; see \Cref{Figure_error_evolution_ACA_functions} (right).
\end{example} 
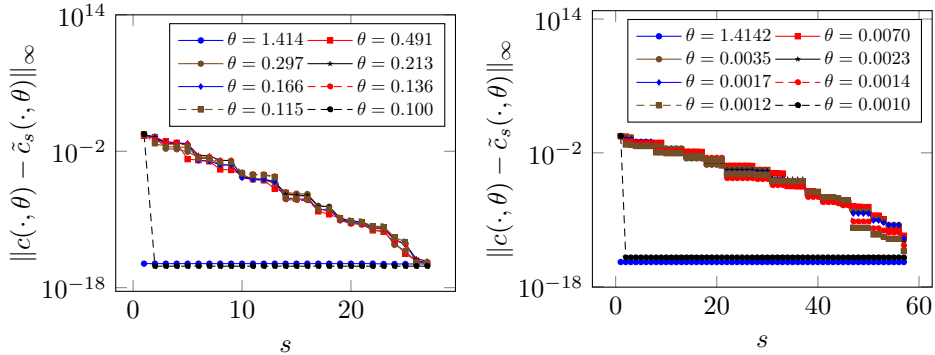
\begin{figure}
	\centering
	\begin{minipage}{.49\linewidth}
		\begin{tikzpicture}
		\begin{semilogyaxis}[
		xlabel = $s$, ylabel=$\norm{c(\cdot,\param)-\tilde c_s(\cdot,\param)}_\infty$,
		width=.98\linewidth,legend style={nodes={scale=0.7, transform shape}}, legend pos = north east,legend columns=2,  ymax = 1e14]
		\addplot+[mark options={scale=0.5}]  table[x index = 0, y index = 1] {test_aca_fun1.dat};
		\addplot+[mark options={scale=0.5}]  table[x index = 0, y index = 2] {test_aca_fun1.dat};
		\addplot+[mark options={scale=0.5}]  table[x index = 0, y index = 3] {test_aca_fun1.dat};
		\addplot+[mark options={scale=0.5}]  table[x index = 0, y index = 4] {test_aca_fun1.dat};
				\addplot+[mark options={scale=0.5}]  table[x index = 0, y index = 5] {test_aca_fun1.dat};
		\addplot+[mark options={scale=0.5}]  table[x index = 0, y index = 6] {test_aca_fun1.dat};
		\addplot+[mark options={scale=0.5}]  table[x index = 0, y index = 7] {test_aca_fun1.dat};
		\addplot+[mark options={scale=0.5}]  table[x index = 0, y index = 8] {test_aca_fun1.dat};
		\legend{$\param =1.414$,$\param =0.491$,$\param =0.297$,$\param =0.213$,$\param =0.166$,$\param =0.136$,$\param =0.115$,$\param =0.100$,};
		\end{semilogyaxis}
		\end{tikzpicture}    
	\end{minipage}~\begin{minipage}{.49\linewidth}
			\begin{tikzpicture}
	\begin{semilogyaxis}[
	xlabel = $s$,  ylabel=$\norm{c(\cdot,\param)-\tilde c_s(\cdot,\param)}_\infty$,
	width=.98\linewidth,legend style={nodes={scale=0.7, transform shape}}, legend pos = north east,legend columns=2,  ymax = 1e15]
	\addplot+[mark options={scale=0.5}]  table[x index = 0, y index = 1] {test_aca_fun2.dat};
	\addplot+[mark options={scale=0.5}]  table[x index = 0, y index = 2] {test_aca_fun2.dat};
	\addplot+[mark options={scale=0.5}]  table[x index = 0, y index = 3] {test_aca_fun2.dat};
	\addplot+[mark options={scale=0.5}]  table[x index = 0, y index = 4] {test_aca_fun2.dat};
	\addplot+[mark options={scale=0.5}]  table[x index = 0, y index = 5] {test_aca_fun2.dat};
	\addplot+[mark options={scale=0.5}]  table[x index = 0, y index = 6] {test_aca_fun2.dat};
	\addplot+[mark options={scale=0.5}]  table[x index = 0, y index = 7] {test_aca_fun2.dat};
	\addplot+[mark options={scale=0.5}]  table[x index = 0, y index = 8] {test_aca_fun2.dat};
	\legend{$\param =1.4142$,$\param =0.0070$,$\param =0.0035$,$\param =0.0023$,$\param =0.0017$,$\param =0.0014$,$\param =0.0012$,$\param =0.0010$,};
	\end{semilogyaxis}
	\end{tikzpicture} 
	\end{minipage}
 	\caption{Gaussian covariance kernel. Error associated with the approximate covariance kernel  $\tilde c_s$ vs. length of the expansion $s$, in the case  $\param \in [0.1, \sqrt{2}]$ (left) and $\param \in [0.001, \sqrt{2}]$ (right). Each curve represents the infinity norm $\norm{c(\cdot,\param)-\tilde c_s(\cdot,\param)}_\infty$ for a fixed value of $\param$,  approximated by taking the maximum over 500 equispaced values for $d$ in $[0, \sqrt{2}]$.} 
\label{Figure_error_evolution_ACA_functions}
\end{figure}
 
\subsubsection{Multidimensional parameter space}

In the case of a multidimensional $\pset$ we rely on a slight modification of the so called \emph{empirical interpolation method} (EIM) \cite{eim}. Before sketching the procedure, we  recall the notion of \emph{quasimatrix}~\cite{townsend15}.
\begin{definition}
	Let $C([\underline d, \overline d])$ be the set of continuous real functions on $[\underline d, \overline d]$ and $r\in\mathbb N$. An $[\underline d, \overline d]\times r$ quasimatrix is an ordered  $r$-tuple of functions in $C([\underline d, \overline d])$.
	\end{definition}
The name  quasimatrix comes from the representation of a generic tuple $(f_1(d),\dots,f_r(d))$  as a $[\underline d, \overline d]\times r$ matrix $F:=[f_1(d)|\dots|f_r(d)]$ where the continuous row index refers to the argument $d$ while the discrete column index selects the function. Common operations for matrices can be easily extended to quasimatrices. For instance, given a finite set $\mathcal I=\{d_1,\dots,d_s\}\subset [\underline d, \overline d]$  we define   
\[
F(\mathcal I,:):=\begin{bmatrix}
f_1(d_1)&\dots&f_r(d_1)\\
\vdots&&\vdots\\
f_1(d_s)&\dots&f_r(d_s)
\end{bmatrix}\in\mathbb R^{s\times r}.
\]Intuitively, the product $F\cdot M$ with $M\in\mathbb R^{r\times p}$ is defined as the $[\underline d, \overline d]\times p$ quasimatrix whose columns are linear combinations of the columns of $F$. Moreover, the most used matrix factorizations, e.g. LU, QR and SVD, can be extended to quasimatrices \cite{townsend15}. 
 Analogously, we define $r$-tuple of real functions of $\bm{\param}$ as transposed quasimatrices, i.e.  $r\times \pset$ matrices, and the corresponding matrix operations.
 The algorithm that we employ for computing the affine linear expansion \eqref{eq:affinelinear} combines EIM with the SVD of quasimatrices. It proceeds as follows:
\begin{enumerate}
	\item Generate samples  $\bm\param_1,\dots,\bm\param_{r}$ from $\pset$.
	\item Define the snapshot functions $\hat c_j(d):=\tilde c(d,\bm\param_j)$, $j=1,\dots,r$ and form the quasi-matrix
	\[
	F=\begin{bmatrix}
	\hat c_1(d)|\dots|\hat c_r(d)
	\end{bmatrix}\in \mathbb R^{[\underline d, \overline d]\times r}.
	\]
	\item Compute the SVD of  $F$:
	\[
	F = U\Sigma V^*,\quad U\in \mathbb R^{[\underline d, \overline d]\times r}, \quad \Sigma,V\in\mathbb R^{r\times r}.
	\]
Truncate the SVD decomposition by keeping the first $s$ columns of $u_1(d),\dots, u_s(d)$ of $U$,  which corresponds to the diagonal entries in $\Sigma$ that are above a prescribed tolerance $\tau$.
	\item  Determine a set $\mathcal I=\{d_1,\dots,d_s\}$ of nodes for interpolation by using the greedy strategy of EIM, reported in \Cref{alg:index-sel}.
	\item Set $a_j(\mathbf x,\mathbf y):= u_j(\|\mathbf x-\mathbf y\|_2)$, $j=1,\dots s$ and define the functions $\varphi_j(\bm\param)$ via the transposed quasimatrix which solves the linear system
	\[
	\begin{bmatrix}
	 u_1(d_1)&\dots& u_s(d_1)\\
	 \vdots&\cdots&\vdots\\
	 	 u_1(d_s)&\dots& u_s(d_s)
	\end{bmatrix}
	\begin{bmatrix}
	\varphi_1(\bm\param)\\
	\vdots\\
	\varphi_s(\bm\param)
	\end{bmatrix}=
	\begin{bmatrix}
	\tilde c(d_1,\bm\param)\\
	\vdots\\
	\tilde c(d_s,\bm\param)
	\end{bmatrix}.
	\]
\end{enumerate}
The whole procedure is reported in \Cref{alg:eim}. The practical implementation makes again use of the toolbox Chebfun2  which provides all the functionalities for handling quasimatrices. 

The sampling at step 1 can be performed via a greedy algorithm, as in \cite{eim}, or with other strategies that aim at mitigating the curse of dimensionality when considering a high dimensional parameter space \cite{sparsegrids,greedy1,lohr}. Since in our numerical tests we do not need to deal with very high dimensional parameter spaces, we just perform a tensorized equispaced sampling of $\pset$. 
The accuracy of \Cref{alg:eim} is tested in the following example.

\begin{example} \label{exam_functi_eim}
	Let us consider the Mat\'ern covariance kernel from \Cref{ex:gauss} with $\text{diam}(D)=\sqrt 2$,    $\param_1\in  [0.1, \sqrt{2}]$ and $\param_2\in  [2.5, 7.5]$.  The accuracy of the separable expansion \eqref{eq:fun-aca} --- computed via \Cref{alg:eim} --- is shown in \Cref{Figure_error_evolution_eim};  each curve represents  the maximum absolute error $\max_{d\in[0,\sqrt 2]}|\tilde c(d,\bm\param) - \tilde{c}_s(d,\bm \param)|$ --- as $s$ increases --- for $\bm\param=(\param_1,\param_2)$ chosen on a $10\times 10$ grid obtained as tensor product of equispaced values on $[0.1, \sqrt{2}]$ and  $[2.5, 7.5]$, respectively. The maximum of the error is approximated by taking its largest value over $500$ equispaced samples of $d$ in $[0,\sqrt 2]$. 
\end{example} 

 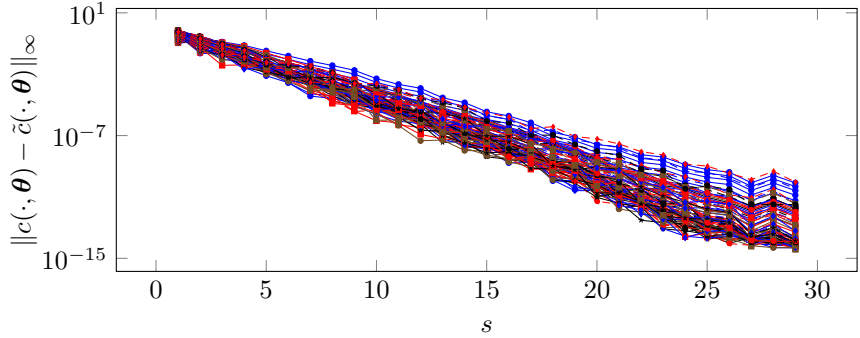
\begin{figure}
	\centering
	\begin{tikzpicture}
\begin{semilogyaxis}[
xlabel = $s$,  ylabel=$\norm{c(\cdot,\bm\param)-\tilde{c}(\cdot,\bm\param)}_\infty$,
width=.9\linewidth,height= 0.4\linewidth,legend style={font=\tiny}]
\foreach \j in {1, ..., 100} {
	\addplot+[mark options={scale=0.5}] table[x index = 0, y index = \j] {test_deim.dat};
}
\end{semilogyaxis}
\end{tikzpicture} 
	\caption{Mat\'ern covariance kernel. Error associated with the approximate covariance kernel  $\tilde c_s$ vs. length of the expansion $s$, for  $(\param_1,\param_2)\in [0.1, \sqrt{2}]\times[2.5, 7.5]$. Each curve represents the infinity norm $\norm{c(\cdot,\bm\param)-\tilde c_s(\cdot,\bm\param)}_\infty$ for a fixed pair $\bm\param=(\param_1,\param_2)$,  approximated by taking the maximum over 500 equispaced values for $d$ in $[0, \sqrt{2}]$.}
	\label{Figure_error_evolution_eim}
\end{figure}
\begin{algorithm}
	\caption{Interpolation nodes selection}\label{alg:index-sel}
	\begin{algorithmic}[1]
		\Statex{\textbf{procedure}} node\_selection($u_1(d),\dots,u_K(d)$)
		\Comment{real functions  defined on $[\underline d, \overline d]$}
		\State Set $d_1= \arg\max_{[\underline d, \overline d]} |u_1(d)|$,\ \ \ $\mathcal I = \{d_1\}$,\ \ \ $U = [u_1(d)]$ 
		\For{$j:=2,\dots,s $} 
		\State $c \gets U(\mathcal I, 1:j-1)^{-1}u_j(\mathcal I)$   
		\State $r(d)\gets u_j(d)- U\cdot c$ 
		\State $d_j\gets \arg\max_{[\underline d, \overline d]}|r(d)|$
		\State $U\gets [U, u_j(d)]$,\ \ \ $\mathcal I\gets \mathcal I\cup \{d_j\}$
		\EndFor
		\State \textbf{return} $\mathcal I$ 
	\end{algorithmic}
\end{algorithm}
\begin{algorithm}
	\caption{Empirical Interpolation method}\label{alg:eim}
	\begin{algorithmic}[1]
		\Statex{\textbf{procedure}} EIM($\tilde c(d, \bm\param)$, $\tau$)
		\Comment{$\tilde c$ real function  defined on $[\underline d, \overline d]\times\pset $}
		\State Sample $\bm\param_1,\dots,\bm\param_r$ from $\pset$ 
		\State Set $F=[\tilde c( d,\bm\param_1)|\dots|\tilde c(d, \bm\param_r)]\in\mathbb R^{[\underline d, \overline d]\times r}$
		\State $[U, \Sigma, V] \gets \text{svd}(F)$
		\State Compute $s$ s.t. $\Sigma_{s,s}>\tau>\Sigma_{s+1,s+1}$
		\State $\mathcal I\gets \text{node\_selection}(U(:,1),\dots,U(:,s))$
		\State $U_s\gets U(\mathcal I, 1:s)$
		\For{$j=1,\dots s$}
		\State $a_j(\bm x,\bm y)\gets U(\|\bm x-\bm y\|_2, j)$
		\State $\varphi_j(\bm\param)\gets e_j^T U_s^{-1}\tilde c(\mathcal I, \bm\param)$\Comment{$e_j$ $j$-th vector of the canonical basis}
		\EndFor
		\State \textbf{return} $a_1,\dots,a_s, \varphi_1,\dots,\varphi_s$
	\end{algorithmic}
\end{algorithm}

\section{Numerical experiments} \label{Sec_NumExp}

We test the performance of \Cref{alg:param-aca} for the examples discussed in the previous section.
All experiments have been carried out on a Laptop with a dual-core Intel Core i7-7500U 2.70 GHz CPU, 256KB of level 2 cache, and 16 GB of RAM. The
algorithms are implemented in Matlab and tested under MATLAB2019a, 
with MKL BLAS version 2018.0.3 utilizing both cores.
\subsection{Parameterized  Gaussian covariance kernel} \label{Subsec_Num_exp_GRF}

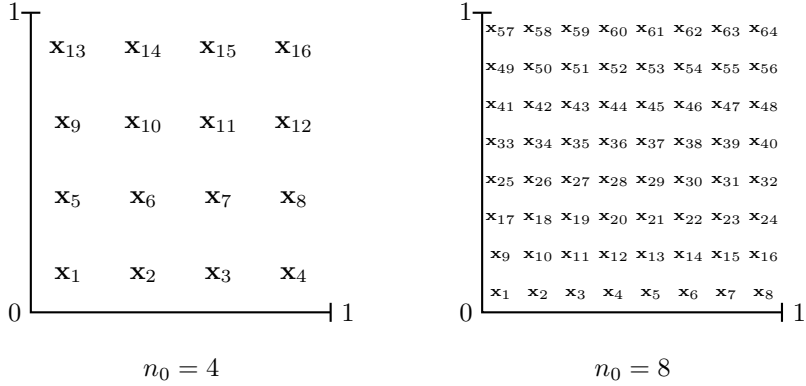
\begin{figure}[thb]
\centering
\begin{tikzpicture}[scale=1]
\draw [|-|,thick] (0.5,4.5) node (yaxis) [left] {$1$}
        |- (4.5,0.5) node (xaxis) [right] {$1$};
        \draw (0.5,0.5) node [left] {$0$};
        \draw (0.75,0.75) node {\tiny $\mathbf{x}_1$};
        \draw (1.25,0.75) node {\tiny $\mathbf{x}_2$};
        \draw (1.75,0.75) node {\tiny $\mathbf{x}_3$};
        \draw (2.25,0.75) node {\tiny $\mathbf{x}_4$};
        \draw (2.75,0.75) node {\tiny $\mathbf{x}_5$};
        \draw (3.25,0.75) node {\tiny $\mathbf{x}_6$};
        \draw (3.75,0.75) node {\tiny $\mathbf{x}_7$};
        \draw (4.25,0.75) node {\tiny $\mathbf{x}_8$};
        
        \draw (0.75,1.25) node {\tiny $\mathbf{x}_{9}$};
        \draw (1.25,1.25) node {\tiny $\mathbf{x}_{10}$};
        \draw (1.75,1.25) node {\tiny $\mathbf{x}_{11}$};
        \draw (2.25,1.25) node {\tiny $\mathbf{x}_{12}$};
        \draw (2.75,1.25) node {\tiny $\mathbf{x}_{13}$};
        \draw (3.25,1.25) node {\tiny $\mathbf{x}_{14}$};
        \draw (3.75,1.25) node {\tiny $\mathbf{x}_{15}$};
        \draw (4.25,1.25) node {\tiny $\mathbf{x}_{16}$};
        
        \draw (0.75,1.75) node {\tiny $\mathbf{x}_{17}$};
        \draw (1.25,1.75) node {\tiny $\mathbf{x}_{18}$};
        \draw (1.75,1.75) node {\tiny $\mathbf{x}_{19}$};
        \draw (2.25,1.75) node {\tiny $\mathbf{x}_{20}$};
        \draw (2.75,1.75) node {\tiny $\mathbf{x}_{21}$};
        \draw (3.25,1.75) node {\tiny $\mathbf{x}_{22}$};
        \draw (3.75,1.75) node {\tiny $\mathbf{x}_{23}$};
        \draw (4.25,1.75) node {\tiny $\mathbf{x}_{24}$};
                
        \draw (0.75,2.25) node {\tiny $\mathbf{x}_{25}$};
        \draw (1.25,2.25) node {\tiny $\mathbf{x}_{26}$};
        \draw (1.75,2.25) node {\tiny $\mathbf{x}_{27}$};
        \draw (2.25,2.25) node {\tiny $\mathbf{x}_{28}$};
        \draw (2.75,2.25) node {\tiny $\mathbf{x}_{29}$};
        \draw (3.25,2.25) node {\tiny $\mathbf{x}_{30}$};
        \draw (3.75,2.25) node {\tiny $\mathbf{x}_{31}$};
        \draw (4.25,2.25) node {\tiny $\mathbf{x}_{32}$};
        
        \draw (0.75,2.75) node {\tiny $\mathbf{x}_{33}$};
        \draw (1.25,2.75) node {\tiny $\mathbf{x}_{34}$};
        \draw (1.75,2.75) node {\tiny $\mathbf{x}_{35}$};
        \draw (2.25,2.75) node {\tiny $\mathbf{x}_{36}$};
        \draw (2.75,2.75) node {\tiny $\mathbf{x}_{37}$};
        \draw (3.25,2.75) node {\tiny $\mathbf{x}_{38}$};
        \draw (3.75,2.75) node {\tiny $\mathbf{x}_{39}$};
        \draw (4.25,2.75) node {\tiny $\mathbf{x}_{40}$};
                    
        \draw (0.75,3.25) node {\tiny $\mathbf{x}_{41}$};
        \draw (1.25,3.25) node {\tiny $\mathbf{x}_{42}$};
        \draw (1.75,3.25) node {\tiny $\mathbf{x}_{43}$};
        \draw (2.25,3.25) node {\tiny $\mathbf{x}_{44}$};
        \draw (2.75,3.25) node {\tiny $\mathbf{x}_{45}$};
        \draw (3.25,3.25) node {\tiny $\mathbf{x}_{46}$};
        \draw (3.75,3.25) node {\tiny $\mathbf{x}_{47}$};
        \draw (4.25,3.25) node {\tiny $\mathbf{x}_{48}$};
                
        \draw (0.75,3.75) node {\tiny $\mathbf{x}_{49}$};
        \draw (1.25,3.75) node {\tiny $\mathbf{x}_{50}$};
        \draw (1.75,3.75) node {\tiny $\mathbf{x}_{51}$};
        \draw (2.25,3.75) node {\tiny $\mathbf{x}_{52}$};
        \draw (2.75,3.75) node {\tiny $\mathbf{x}_{53}$};
        \draw (3.25,3.75) node {\tiny $\mathbf{x}_{54}$};
        \draw (3.75,3.75) node {\tiny $\mathbf{x}_{55}$};
        \draw (4.25,3.75) node {\tiny $\mathbf{x}_{56}$};
                       
        \draw (0.75,4.25) node {\tiny $\mathbf{x}_{57}$};
        \draw (1.25,4.25) node {\tiny $\mathbf{x}_{58}$};
        \draw (1.75,4.25) node {\tiny $\mathbf{x}_{59}$};
        \draw (2.25,4.25) node {\tiny $\mathbf{x}_{60}$};
        \draw (2.75,4.25) node {\tiny $\mathbf{x}_{61}$};
        \draw (3.25,4.25) node {\tiny $\mathbf{x}_{62}$};
        \draw (3.75,4.25) node {\tiny $\mathbf{x}_{63}$};
        \draw (4.25,4.25) node {\tiny $\mathbf{x}_{64}$};

                \draw (2.5,0) node [below] {$n_0 = 8$};

\draw [|-|,thick] (-5.5,4.5) node (yaxis) [left] {$1$}
        |- (-1.5,0.5) node (xaxis) [right] {$1$};
        \draw (-5.5,0.5) node [left] {$0$};
        \draw (-5,1) node {$\mathbf{x}_1$};
        \draw (-4,1) node {$\mathbf{x}_2$};
        \draw (-3,1) node {$\mathbf{x}_3$};
        \draw (-2,1) node {$\mathbf{x}_4$};
                \draw (-5,2) node {$\mathbf{x}_5$};
        \draw (-4,2) node {$\mathbf{x}_6$};
        \draw (-3,2) node {$\mathbf{x}_7$};
        \draw (-2,2) node {$\mathbf{x}_8$};
                        \draw (-5,3) node {$\mathbf{x}_9$};
        \draw (-4,3) node {$\mathbf{x}_{10}$};
        \draw (-3,3) node {$\mathbf{x}_{11}$};
        \draw (-2,3) node {$\mathbf{x}_{12}$};
                     \draw (-5,4) node {$\mathbf{x}_{13}$};
        \draw (-4,4) node {$\mathbf{x}_{14}$};
        \draw (-3,4) node {$\mathbf{x}_{15}$};
        \draw (-2,4) node {$\mathbf{x}_{16}$};
                \draw (-3.5,0) node [below] {$n_0 = 4$};
                
\end{tikzpicture}
\caption{Schematic of the definitions of the nodes $\mathbf{x}_{1}, \ldots, \mathbf{x}_{n}$, for $n_0 \in \{4,8\}$.} 
\label{fig_nodes}
\end{figure}
Let us consider  a hierarchical Gaussian random field on $D=[0,1]^2$ with mean $\bf 0$ and covariance operator induced by the 
 Gaussian covariance kernel $c(\mathbf x,\mathbf y; \param)$ defined in \eqref{eq:gauss-ker} with $\param \in \pset = [0.1, \sqrt{2}]$. The finite surrogate $\pset_f$  of the parameter space is obtained by taking  equispaced points on $\pset$.
The approximate linearized kernel $\widetilde c_s$ is retrieved via the function-valued ACA as  in  \Cref{exam_functi_ACA_Gauss}. The  cut-off value is set to $s = 18$ which provides, according to the results shown in \Cref{Figure_error_evolution_ACA_functions}, a cut-off error of about $10^{-8}$. 
 
We discretize the random field on a regular grid with $n:= n_0^2$ grid points. 
In turn, the matrix $\mathbf{C}(\theta) \in \mathbb{R}^{n \times n}$ is given by

\begin{align*}
\mathbf{C}(\theta)_{i,j} &:= \frac{1}{n} \cdot c\left(\mathbf x_i,\mathbf x_j; \param\right)  &(i, j = 1,\ldots,n),
\\ \mathbf x_i &:= \left(\frac{\mathrm{mod}(i,n_0) + 0.5}{n_0 +1}, \frac{\mathrm{div}(i,n_0) + 0.5}{n_0 +1}\right)^T &(i = 1,\ldots,n),
\end{align*}
where $\mathrm{div}(i,n_0)$ represents the integer division of $i$ by $n_0$ and $\mathrm{mod}(i,n_0)$ represents the associated remainder term. 
We illustrate the node positions with associated indices in \Cref{fig_nodes}. Note that this covariance matrix corresponds to a finite element discretization of the random field using piecewise constant ansatzfunctions with square supports and quadrature with one node in the centre of each of the supports.

\begin{remark}[BTTB] 
Note that the matrix $\mathbf{C}(\theta)$ is \emph{Block Toeplitz with Toeplitz Blocks} (BTTB). 
This is caused by the regular distribution of the nodes and the isotropy of the covariance kernel.
 For covariance matrices with this structure, efficient low-rank approximation \cite{Halko} and sampling strategies (circulant embedding) are available. 
Note that, our ACA method does neither require this structure nor profit from it. Thus, the results discussed in the following extend to covariance matrices that are not BTTB which arise --- for instance --- when considering more complex geometries $D$ or irregular discretizations.
\end{remark}

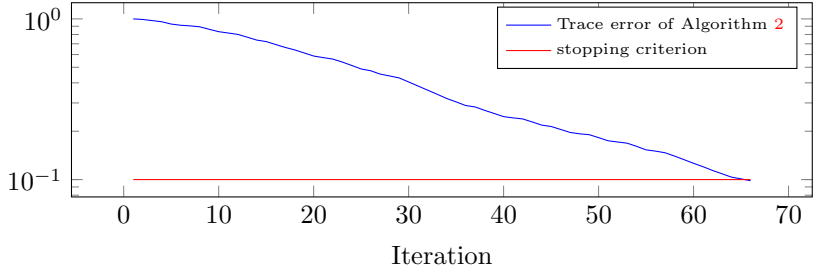
\begin{figure}
	\centering
	\begin{tikzpicture}
	\begin{semilogyaxis}[
	xlabel = Iteration, 
	width=0.9\linewidth,height=0.2\textheight,legend style={font=\tiny}, legend cell align={left}]
	\addplot[blue] table[x index = 0, y index = 1] {gaussian_convergence.dat};
	\addplot[red] table[x index = 0, y index = 2] {gaussian_convergence.dat};
	\legend{ Trace error of \Cref{alg:param-aca}, stopping criterion};
	\end{semilogyaxis}
	\end{tikzpicture}
 \caption{Test 1. Error vs. number of iterations for \Cref{alg:param-aca} applied to a discretized Gaussian covariance operator.}
\label{Figure_Time_evolution_ACA}
\end{figure}
\begin{paragraph}{Test 1}
As a first test, we set $n_0 = 512$, i.e. $n \approx 2.6\cdot 10^5$ degrees of freedom and the cardinality of $\pset_f$  to $1000$. Then, we run \Cref{alg:param-aca} using $\mathrm{tol} = 10^{-1}$ as value for the tolerance at line~\ref{step:stop} of \Cref{alg:aca}. This tolerance guarantees that a KL expansion with the same Wasserstein(2) error would represent 90\% of the random field's variance; this is a usual choice in random field discretizations (see \Cref{remark_KL_Was}).
 
The evolution of the trace error as the algorithm proceeds is shown in \Cref{Figure_Time_evolution_ACA}.
The algorithm terminates after $65$ iterations meaning that a 65-dimensional basis is sufficient to represent the parameter-dependent $512^2 \times 512^2$ operator with the desired accuracy.
Note that \Cref{alg:param-aca} always picks the smallest correlation length $\theta^*  = 0.1$ in the for-loop beginning in line~\ref{step:inner-loop}. 

The computational time of the algorithm is $336.8$ seconds. We remark that the updating strategy for the QR factorization at line~\ref{line:updateRI} of \Cref{alg:param-aca}, described in \Cref{subsec_ACA_param}, is crucial. Indeed, computing the QR factorization from scratch in each iteration increases the execution of \Cref{alg:param-aca} to $1593.3$ seconds; $80\%$ of which are spent on line~\ref{line:updateRI}.
\end{paragraph}

\begin{paragraph}{Test 2}
	In order to shed some light on the role played by $\param$,  \Cref{Figure_eigenvaluedecay} shows the sorted eigenvalues of $\mathbf{C}(\theta)$ for $\theta \in [0.1,\sqrt{2}]$ when using $n_0 = 128$ for the discretization. 
We conclude that a smaller correlation length leads to a slower eigenvalue decay. Consequently, a small value of $\param$ implies that a larger rank in a low-rank approximation of $\mathbf{C}(\param)$ is needed in order to retain the same accuracy.
\begin{figure}[hptb]
  	\centering
  	\includegraphics[scale=0.55]{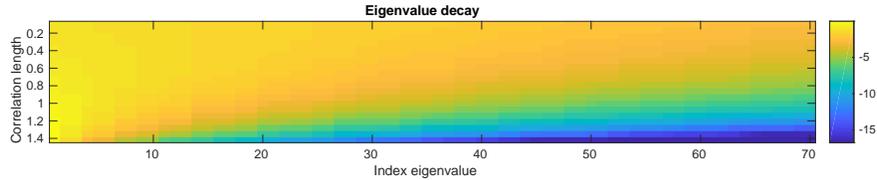}
  	\caption{Test 2. Decimal logarithms of the largest 70 eigenvalues $\mathbf{C}(\theta)$ for different values of $\theta \in [0.1,\sqrt{2}]$ and $n_0 = 128$.}
  	\label{Figure_eigenvaluedecay}
  \end{figure}
\end{paragraph}

\begin{paragraph}{Test 3} We have measured the execution time of \Cref{alg:param-aca} for the example described above with $n_0 \in \{2^3,\ldots,2^9\}$ and $|\pset_f| \in \{10, 10^2, 10^3, 10^4\}$.  The timings shown in \Cref{Figure_timings_ranks} (left) are averaged over three runs.
It can be seen that  $n$ influences the cost linearly, as predicted by the complexity estimates derived in \Cref{subsec_ACA_param}.
Moreover, we reported the rank  of the ACA approximation, i.e. the cardinality of the index set $I$ returned by \Cref{alg:param-aca}, in \Cref{Figure_timings_ranks} (right). We see that the rank converges to $65$, when refining the random field, suggesting
that the rank obtained with ACA might be discretization invariant.
Also the rank seems to not be influenced at all by the cardinality of $\pset_f$ since the algorithm always chooses the smallest correlation length. 
In this particular setting, the outcome of the algorithm does not change, as long as $\underline{\param}$ does not vary.
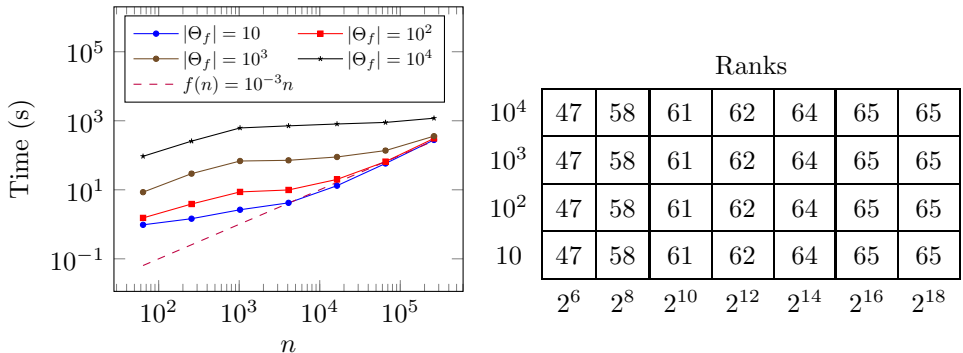
\begin{figure}
	\centering
	\begin{minipage}{.49\linewidth}
		\begin{tikzpicture}
		\begin{loglogaxis}[
		xlabel = $n$, ylabel = Time (s), 
		width=\linewidth,legend style={nodes={scale=0.7, transform shape}}, legend pos = north west,legend columns=2,  ymax = 2e6, legend cell align={left}]
		\addplot+[mark options={scale=0.5}]  table[x index = 0, y index = 1] {test_cpu_time.dat};
		\addplot+[mark options={scale=0.5}]  table[x index = 0, y index = 2] {test_cpu_time.dat};
		\addplot+[mark options={scale=0.5}]  table[x index = 0, y index = 3] {test_cpu_time.dat};
		\addplot+[mark options={scale=0.5}]  table[x index = 0, y index = 4] {test_cpu_time.dat};
		 \addplot[domain=64 : 250000, dashed, purple] {x/1000};
		\legend{$|\pset_f|=10$, $|\pset_f|=10^2$,$|\pset_f|=10^3$,$|\pset_f|=10^4$, $f(n)=10^{-3}n$};
		\end{loglogaxis}
		\end{tikzpicture}    
	\end{minipage}~\begin{minipage}{.49\linewidth}
	\renewcommand*{\arraystretch}{1.5}
		\begin{tabular}{cccccccc}
			\multicolumn{1}{l}{}        & \multicolumn{7}{c}{Ranks}                                                                                                                                                           \\ \cline{2-8} 
			\multicolumn{1}{c|}{$10^4$} & \multicolumn{1}{c|}{47} & \multicolumn{1}{c|}{58} & \multicolumn{1}{c|}{61} & \multicolumn{1}{c|}{62} & \multicolumn{1}{c|}{64} & \multicolumn{1}{c|}{65} & \multicolumn{1}{c|}{65} \\ \cline{2-8} 
			\multicolumn{1}{c|}{$10^3$} & \multicolumn{1}{c|}{47} & \multicolumn{1}{c|}{58} & \multicolumn{1}{c|}{61} & \multicolumn{1}{c|}{62} & \multicolumn{1}{c|}{64} & \multicolumn{1}{c|}{65} & \multicolumn{1}{c|}{65} \\ \cline{2-8} 
			\multicolumn{1}{c|}{$10^2$} & \multicolumn{1}{c|}{47} & \multicolumn{1}{c|}{58} & \multicolumn{1}{c|}{61} & \multicolumn{1}{c|}{62} & \multicolumn{1}{c|}{64} & \multicolumn{1}{c|}{65} & \multicolumn{1}{c|}{65} \\ \cline{2-8} 
			\multicolumn{1}{c|}{$10$}   & \multicolumn{1}{c|}{47} & \multicolumn{1}{c|}{58} & \multicolumn{1}{c|}{61} & \multicolumn{1}{c|}{62} & \multicolumn{1}{c|}{64} & \multicolumn{1}{c|}{65} & \multicolumn{1}{c|}{65} \\ \cline{2-8} 
			\multicolumn{1}{c}{}        & $2^6$                   & $2^8$                   & $2^{10}$                & $2^{12}$                & $2^{14}$                & $2^{16}$                & $2^{18}$               
		\end{tabular}

	\end{minipage}
	\caption{Test 3. Timings in seconds (left) and ranks (right) obtained from the ACA algorithm applied to different discretizations of the covariance operator. 
		The x-axes of both figures show the degrees of freedom $n = n_0^2$ of the discretized covariance operators. The y-axes show the mean elapsed time over three runs for the figure on the left and the quantity $|\pset_f|$ for the table on the right.}
	\label{Figure_timings_ranks}
\end{figure}
\end{paragraph}
 \begin{paragraph}{Test 4}
 Now, we compare the convergence history of the method with the nuclear norm of the error and the trace error of the best low-rank approximations of  $\mathbf{C}(\param)$. We remark that the convergence history of \Cref{alg:param-aca} corresponds to the trace error with respect to the approximate covariance kernel, i.e.  the quantity $\max_{\param\in\pset_f}\trace(\mathbf{C}_s(\param)- \mathbf{ C}_{sI}(\param))$. Since the approximate covariance kernel is not guaranteed to be SPSD, as true measure of the error we consider  the maximum nuclear norm of $\mathbf{C}(\param)- \mathbf{ C}_{sI}(\param)$ over $\param\in\pset_f$. Finally, since $\mathbf{ C}_{sI}(\param)$ is a rank $k=|I|$ matrix for each value of $\param$, we also compute the benchmark quantity $\max_{\param\in\pset_f}\trace(\mathbf{C}(\param)- \mathbf{C}_{(k)}(\param))$, where   $\mathbf{C}_{(k)}(\param)$ denotes the truncated SVD of length $k$ of $\mathbf{C}(\param)$.

 We run Algorithm~\ref{alg:param-aca} on an  example of small dimension: $n_0 = 20$ and $|\pset_f|=200$. Then, we plot the three quantities discussed above in Figure~\ref{Figure_error_comparison}.  We see that the trace error computed by Algorithm~\ref{alg:param-aca} and the nuclear norm of the error coincide as far as  they stay above $10^{-8}$. Then, the nuclear norm of the error stagnates because we have reached the level of inexactness which affects the approximate linear separable expansion. To conclude, we observe that
 the trace error associated with the truncated SVDs decays only slightly faster confirming the good convergence rate of the approximation returned by Algorithm~\ref{alg:param-aca}. 

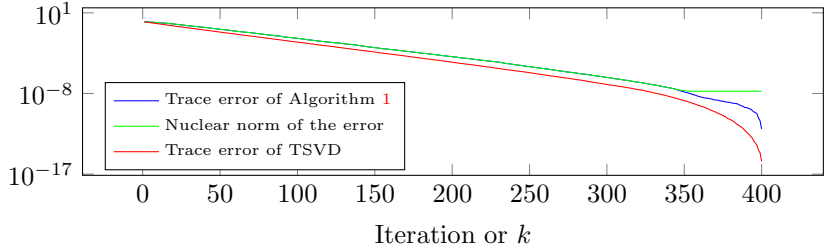
\begin{figure}
	\centering
\begin{tikzpicture}
\begin{semilogyaxis}[
xlabel = Iteration or $k$, 
width=0.9\linewidth, height=0.3\linewidth,legend style={font=\tiny}, legend pos = south west, legend cell align={left}]
\addplot[blue] table[x index = 0, y index = 1] {test_tensor_rank.dat};
\addplot[green] table[x index = 0, y index = 2] {test_tensor_rank.dat};
\addplot[red] table[x index = 0, y index = 3] {test_tensor_rank.dat};
\legend{Trace error of \Cref{alg:aca}, Nuclear norm of the error, Trace error of TSVD};
\end{semilogyaxis}
\end{tikzpicture}
\caption{Test $4$. Comparison between the convergence history of \Cref{alg:aca} (blue),  the maximum nuclear norm of $\mathbf{C}(\param)- \mathbf{ C}_{sI}(\param)$ over $\pset_f$ (green) and the best trace error associated with the truncated SVDs $\max_{\param\in\pset_f}\trace(\mathbf{C}(\param)- \mathbf{C}_{(k)}(\param))$ (red).}
\label{Figure_error_comparison}
\end{figure}
\end{paragraph}
\begin{paragraph}{Test 5}
We conclude the experiments on the Gaussian kernel by testing the sampling strategy proposed in \Cref{sec:sampling} for $n_0=512$, and $|\pset_f|=100$. More specifically, we consider the task of sampling $u$ values of $\param$  from a uniform distribution on $[0.1,\sqrt 2]$ and --- for each of those values --- generating one sample from $\mathrm{N}(0, \Cb(\param))$. Our strategy is to run \Cref{alg:param-aca} to get the approximated kernel operator $\Cb_I(\cdot)$ and then sample from $\mathrm{N}(0, \Cb_I(\param))$ as described in \Cref{sec:sampling}.  We compare this method with applying \Cref{alg:aca} directly on $\Cb(\param)$ for each sample of $\param$. The timings of the two approaches, as the number of samples increases, are reported in \Cref{Figure_sampling}. Notice that, the time consumption of the strategy based on \Cref{alg:param-aca} starts from a positive value which indicates the cost of the offline phase. We see that, while both methods have linear costs with repsect to $u$, the online phase of our strategy is faster by a factor about $4.6$ than generating one sample with \Cref{alg:aca}. In our test, about $200$ samples are enough to amortize the cost of the offline phase and to make our algorithm more convenient.
\begin{figure}
	\centering
	\begin{tikzpicture}
	\begin{axis}[
	xlabel = $u$, ylabel = Time (s),legend pos = north west,
	width=0.9\linewidth,height=0.4\linewidth,legend style={font=\tiny}, legend cell align={left}]
		\addplot[red] table[x index = 0, y index = 2] {test_sampling.dat};
	\addplot[blue] table[x index = 0, y index = 1] {test_sampling.dat};
	\legend{Sampling with \Cref{alg:aca}, Sampling with \Cref{alg:param-aca}};
	\end{axis}
	\end{tikzpicture}
	\caption{Test 5. Time consumption vs. number of samples for the sampling strategies based on \Cref{alg:aca}  and \Cref{alg:param-aca}, applied to a discretized Gaussian covariance operator.}
	\label{Figure_sampling}
\end{figure}
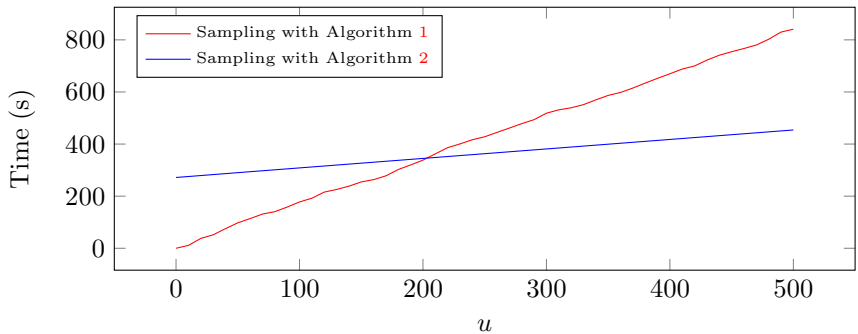
\end{paragraph}
 \subsection{Parameterized Mat\'ern covariance kernel with slower eigenvalue decay} \label{Subsec_Num_exp_MAT}
In this section, we modify the parameterized random field considered in \Cref{Subsec_Num_exp_GRF} by replacing the Gaussian covariance kernel with the Mat\'ern covariance kernel defined in \eqref{eq:matern} with $\param_1 \in [0.1,\sqrt 2]$ and $\param_2 = 2.5$. Also here, we make use of an approximate linearized kernel $\widetilde c_K$ --- computed via the function-valued ACA ---  with  cut-off value  $s = 18$ which again provides an approximation error of about $10^{-8}$. This ensures that the approximation error of $\tilde c_s$ does not affect the convergence history of \Cref{alg:param-aca} in the considered example. 

The  samples with this covariance kernel are only twice continuously differentiable, as opposed to the analytic samples from the Gaussian covariance kernel. 
Hence, we expect a larger  numerical rank of the associated covariance matrix $\mathbf{C}(\param)$ and a larger rank when employing the ACA algorithm.

 We run \Cref{alg:param-aca} with  $n = n_0^2 = 512^2$, $|\pset_f|=1000$ and $\mathrm{tol}=10^{-1}$. 
The evolution over the iterations of the trace error computed by \Cref{alg:param-aca} is reported in \Cref{Figure_Time_evolution_nu25}.
  The algorithm terminates after 106 iterations, and it takes $1015.9$ seconds, averaged over a total of three runs. It thus takes about $3$ times longer than a similar experiment with the Gaussian kernel.  
%
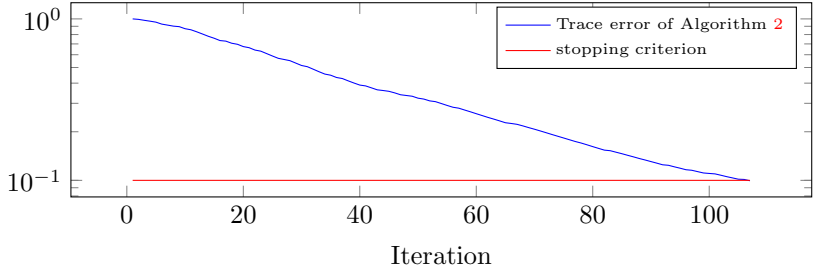
\begin{figure}
	\centering
	\begin{tikzpicture}
	\begin{semilogyaxis}[
	xlabel = Iteration, 
	width=0.9\linewidth,height=0.2\textheight,legend style={font=\tiny}, legend cell align={left}]
	\addplot[blue] table[x index = 0, y index = 1] {matern_convergence.dat};
	\addplot[red] table[x index = 0, y index = 2] {matern_convergence.dat};
	\legend{ Trace error of \Cref{alg:param-aca}, stopping criterion};
	\end{semilogyaxis}
	\end{tikzpicture}
\caption{Error vs. number of iterations for \Cref{alg:param-aca} applied to a discretized 2D Mat\'ern kernel with $\param_1\in[0.1,\sqrt 2]$ and $\param_2 = 2.5$.}
 \label{Figure_Time_evolution_nu25}
\end{figure}
As expected, the rank that we obtain is indeed larger than before. On the other hand, a 106-dimensional basis is still suitably small to efficiently represent a parameter dependent matrix of size $512^2 \times 512^2$. 

\section{Conclusions} \label{Sec_Concl}
We have proposed and analyzed a new method for the low-rank approximation of the Cholesky factor of a parameter dependent covariance operator. The approximation returned by our algorithm is certified, in the sense that it guarantees an upper bound for the error in the Wasserstein distance $W_2$ between mean-zero Gaussian measures having the true and low-rank covariance matrix, respectively.
\Cref{alg:param-aca} leads naturally to a fast sampling procedure for parameterized Gaussian random fields and potentially for the efficient treatment of certain Bayesian inverse problems \cite{Dunlop2017,LATZ2019FastSa}.
Beyond Gaussian random fields, the parameter-dependent ACA method can be applied in much more general hierarchical, kernel-based learning tasks. 
Consider, e.g., a non-linear support vector machine with parameterized (i.e. shallow),  isotropic, symmetric, and  positive semi-definite kernels that shall be trained with a large amount of data. 
Here, the ACA leads to a natural compression of the kernel matrix. 
Analysing the error introduced by ACA in other learning tasks would be an interesting topic for future work.

There is also other potential for future work. For example, while we have observed experimentally, for the examples tested, that \Cref{alg:param-aca} returns an index set $I$ of low cardinality that yields a good approximation $A_I(\bm\param)$ for all parameter values, it would be interesting to derive a prior existence results based on properties of the kernel.




\bibliographystyle{AIMS}
\bibliography{library}

\medskip
Received xxxx 20xx; revised xxxx 20xx.
\medskip

\end{document}